\documentclass{amsart}

\usepackage[english]{babel}
\usepackage[utf8]{inputenc}
\usepackage[numbers,sort&compress]{natbib}
\usepackage{comment,placeins,xcolor}
\usepackage{latexsym,amsfonts,amsmath,graphics,amssymb}
\usepackage{verbatim,url,mathscinet}
\usepackage{epsfig,psfrag,fullpage,graphicx}

\title{A quasi-Monte Carlo data compression algorithm for machine learning}
\thanks{MF is funded by the Deutsche Forschungsgemeinschaft (DFG, German Research Foundation) -- Project-ID 258734477 -- SFB 1173 as well as the Austrian Science Fund (FWF)
under the special research program Taming complexity in PDE systems (grant SFB F65).}

\author{Josef Dick
\and Michael Feischl} 
\def\R{{\mathbb R}}

\def\N{{\mathbb N}}
\def\Z{{\mathbb Z}}

\def\ba{{\boldsymbol{a}}}
\def\bd{{\boldsymbol{d}}}
\def\bi{{\boldsymbol{i}}}
\def\bz{{\boldsymbol{z}}}
\def\bx{{\boldsymbol{x}}}
\def\by{{\boldsymbol{y}}}
\def\bk{{\boldsymbol{k}}}

\def\bomega{\boldsymbol{\omega}}

\def\rd{{\rm d}}

\def\NN{{\mathcal N}}

\def\XX{{\mathcal X}}
\def\YY{{\mathcal Y}}

\newcommand{\dist}[3][]{{\rm d\!l}[\ifthenelse{\equal{#1}{}}{}{#1;}#2,#3]}
\newcommand{\eff}[3][]{{\rm osc}\ifthenelse{\equal{#1}{}}{}{_{#1}}(#2;#3)}

\def\norm#1#2{\|#1\|_{#2}}

\def\set#1#2{\big\{#1\,:\,#2\big\}}

\def\eps{\varepsilon}

\def\b#1{{\boldsymbol{#1}}}

\def\normL2#1#2{\|#1\|_{L^2(#2)}}

\newtheorem{theorem}{Theorem}

\newtheorem{lemma}[theorem]{Lemma}
\newtheorem{corollary}[theorem]{Corollary}
\newtheorem{algorithm}[theorem]{Algorithm}

\newtheorem{remark}[theorem]{Remark}

\newcommand{\satop}[2]{\stackrel{\scriptstyle{#1}}{\scriptstyle{#2}}}

\date{\today}

\begin{document}

\maketitle

\begin{abstract}
We introduce an algorithm to reduce large data sets using so-called digital nets, which are well
 distributed point sets in the unit cube. These point sets together with weights, which depend on the data set, are used to represent the data. We show that this can be used to reduce the computational effort needed in finding good parameters in machine learning algorithms. To illustrate our method we provide some numerical examples for neural networks.
\end{abstract}

% REQUIRED
%\begin{keywords}

{\footnotesize Keywords:  quasi-Monte Carlo, big data, statistical learning, higher-order methods}
%\end{keywords}

% REQUIRED
%\begin{AMS}
{\footnotesize MSC: 65C05, 65D30, 65D32}
%\end{AMS}

\section{Introduction}

Let $\mathcal{X} = \{\bx_1, \ldots, \bx_N\} \subset [0,1]^s$ be a set of data points (given as column vectors) and 
let $\mathcal{Y} = \{ y_1, \ldots, y_N \} \subset \mathbb{R}$ be the corresponding responses 
($y_n$ is the response to $\bx_n$). We want to find a predictor $f_\theta:[0,1]^s \to \mathbb{R}$, 
parameterized by $\theta$, such that $f_\theta(\bx_n) \approx y_n$ for $n = 1, 2, \ldots, N$. 
In the simplest case, $f_\theta$ is linear, $f_\theta(\bx) = [1, \bx^\top] \theta$, where $\bx \in [0,1]^s$, 
$\theta \in \mathbb{R}^{s+1}$
are column vectors and $\theta$ needs to be computed from the data. 
Many other 'supervised' machine learning algorithms fall into this category, for instance, neural networks or 
support vector machines, see \cite{HTF17} for a range of other methods.

We consider the case where the quality of our predictor $f_\theta$ is measured by the loss
\begin{equation}\label{eq:errexp}
\mathrm{err}(f_\theta) = \frac{1}{N}\sum_{n=1}^N (f_\theta(\bx_n) - y_n)^2 = \frac{1}{N}\sum_{n=1}^N f_\theta^2(\bx_n) - \frac{2}{N}\sum_{n=1}^N y_n f_\theta(\bx_n) + \frac{1}{N}\sum_{n=1}^N y_n^2,
\end{equation}
with the goal to choose the parameters $\theta$ of the function $f_\theta$ such that $\mathrm{err}(f_\theta)$ is 
'small'. If the optimization procedure is non-trival ((stochastic) gradient descent, Newton's method, {\ldots}), it is possible that $\mathrm{err}(f_\theta)$ (or possibly {$\nabla_\theta^k \mathrm{err}(f_\theta)$ for $k=1,2,\ldots$}) has to be evaluated many times which leads to a cost proportional to
\begin{align*}
 \# \text{optimization steps}\times \underbrace{\text{amount of data}}_{N}.
\end{align*}
In modern applications in big data and machine learning, the number of data entries $N$ can be substantial. The goal of this work is to find a useful compression of the data $\XX$ which still allows us to compute $\mathrm{err}(f_\theta)$ (and derivatives $\nabla^k\mathrm{err}(f_\theta)$) up to a required accuracy in a fast way. 

The main result can be summarized as follows: There exist a point set $P = \{\bz_\ell\}_{\ell=0}^{L-1}$ and weights $\{W_{\XX,P,\nu,\ell} \}_{\ell = 0}^{L-1}, \{ W_{\XX,\YY,P,\nu,\ell} \}_{\ell=0}^{L-1} {\subset} \R$ such that the error $\mathrm{err}(f_\theta)$ of the predictor $f_\theta$ can be approximated by
\begin{align*}
 \mathrm{err}(f_\theta) \approx \mathrm{app}_L(f_\theta) := \sum_{\ell=0}^{L-1} f^2_\theta(\bz_\ell) W_{\XX,P,\nu,\ell} - 2 \sum_{\ell=0}^{L-1} f_\theta(\bz_\ell) W_{\XX,\YY,P,\nu,\ell} + \frac{1}{N}\sum_{n=1}^N y_n^2.
\end{align*}
(The meaning of the parameter $\nu$ will be explained below.) An important feature of this approximation is that the weights $W_{\XX,P, \nu, \ell}$ and $W_{\XX, \YY, P, \nu, \ell}$ do not depend on the parameter $\theta$. We show below that the weights can be computed efficiently in linear cost in $N$ and the convergence of the approximation error is almost linear in $L$ under some smoothness assumptions on $f_\theta$. Under those assumptions, it is reasonable to choose $L\ll N$. Performing the optimization on the approximate quantity $\mathrm{app}_L(f_\theta)$ thus may save considerable computation time, as we now have
\begin{align*}
 \# \text{optimization steps}\times \underbrace{\text{number of representative points}}_{L}.
\end{align*}
This introduces an additional error in the optimization procedure and may lead to a different local minimum 
(often, the optimization problem is not convex, but has many local minima) but the value of the approximate minimum is close to the exact one. We also provide a bound on the distance between the minimum of the square error and the minimum of the approximation of the square error under certain assumptions.

For any iterative optimization method which benefits from good starting values (Newton-Raphson, gradient descent, \ldots),
it is also possible to choose a sequence of values $L_1 \le L_2 \le L_3 \le \cdots$ and 
use $L_k$ in the $k$th optimization step, i.e., we make the approximation $\mathrm{app}_{L_k}(f_\theta)$ 
more accurate as the optimization procedure gets closer to an approximation of the parameters $\theta$.

\subsection{Related literature}
%TODO:

%Lossy compression, lossless compression; support points; subsampling;
A related technique in statistical learning is called \emph{subsampling}, where from a data set which is
too big to be dealt with directly, a subsample is drawn to represent {the} whole data set.
This subsample can be drawn uniformly, or using information available from the data (called \emph{leveraging}).
In~\cite{subsampling1}, an overview over these techniques and their convergence properties is given. The methods
are usually limited to the Monte-Carlo rate of convergence of $L^{-1/2}$ for a cost of $\mathcal{O}(L)$, 
see also~\cite{subsampling2,subsampling3} for linear models as well as~\cite{subsampling4} for more general models.
As a difference to the present method, the subsampling methods do not require the whole data set, but rely on statistical
assumptions about the dataset. Our method needs to run through the whole data set to compress it, but does not pose any restrictions
on the distribution.

Another related method are \emph{support points} introduced in~\cite{spoints}. The method compresses a given distribution
to a finite number of points by solving an optimization problem. These points can then be used to represent the distribution.
The rate of convergence is slightly better than Monte Carlo (by a polylogarithmic factor) but still slower than $\mathcal{O}(L^{-1/2-\eps})$
for all $\eps>0$. The approximation result poses restrictions on the distribution of the data as well as on the functions evaluated
on the data.
In a similar direction, \emph{coresets} summarize a given data in a smaller weighted set of datapoints. Originally developed for computational geometry, it is used for certain learning problems such as $k$-means clustering or PCA in~\cite{coresets}. Sketching algorithms are another way of reducing the size of data sets which is often based on using random projections, see for instance \cite{AAR20} and the references therein.

%\rev{Include sketching algorithms, konnte noch keine gute Referenz finden.}

\bigskip

A similar method as in the present paper can be derived by using sparse grid techniques (see, e.g.,~\cite{sparsegrids} for an overview). The idea
is to approximate $f_\theta$ by its sparse grid interpolation $I_L f_\theta$ and approximate the error using this representation.
While the pros and cons of both approaches must be investigated further, we only mention that the same
duality also appears in the study of high-dimensional integration problems. Both methods have their merits, with a slight
advantage towards quasi-Monte Carlo methods for really high-dimensional problems.

\subsection{Notation}

We introduce some notation used throughout the paper. Let $\mathbb{R}$ be the set of real numbers, $\mathbb{Z}$ be the set of integers, $\mathbb{N}$ be the set of natural numbers and $\mathbb{N}_0$ be the set of non-negative integers. Let $b \ge 2$ be a natural number (later on we will assume that $b$ is a prime number). For a non-negative integer $k$ let $k = \kappa_0 + \kappa_1 b + \cdots + \kappa_{m-1} b^{m-1}$ denote the base $b$ expansion of $k$, i.e., $\kappa_0, \kappa_1, \ldots, \kappa_{m-1} \in \{0, 1, \ldots, b-1\}$. For vectors $\bk = (k_1, k_2, \ldots, k_s)$ we write the base $b$ expansion of $k_j$ as $k_j = \kappa_{j,0} + \kappa_{j,1} b + \cdots + \kappa_{j, m-1} b^{m-1}$.

\subsubsection*{Notation related to (higher order) digital nets}
Given $k,\alpha\in\N$, we define the quantity $\mu_\alpha(k)$ as follows: Let $k=\kappa'_1 b^{a_1-1} + \ldots + \kappa'_r b^{a_r-1}$ for some $r\in\N$, non-zero digits $\kappa'_1, \ldots, \kappa'_r \in \{1, 2, \ldots, b-1\}$, and $1\leq a_r<a_{r-1}<\ldots <a_2<a_1$, i.e., $a_i$ is the position of the $i$th non-zero digit of $k$. Then, we define
\begin{align*}
 \mu_\alpha(k) := a_1 + \ldots + a_{\min\{\alpha,r\}}.
\end{align*}
Further we set $\mu_{alpha}(0) = 0$. For vectors $\bk = (k_1, k_2, \ldots, k_s) \in \N_0^s$ we write $k_j = \kappa'_{j,1} b^{a_{j,1}-1} + \cdots + \kappa'_{j, r_j} b^{a_{j,r_j}-1}$ and set $$\mu_\alpha(\bk) = \mu_\alpha(k_1) + \cdots + \mu_{\alpha}(k_s) = \sum_{j=1}^s \sum_{i=1}^{\min\{\alpha, r_j \} } a_{j, i}.$$ 

For a vector $\boldsymbol{z} \in [0,1)^s$ we write $\boldsymbol{z} = (z_1, z_2, \ldots, z_s)^\top$. We write the base $b$ expansion of the components of a vector as $z_{j} = z_{j,1} b^{-1} + z_{j,2} b^{-2} + \cdots$, where $z_{j,i} \in \{0, 1, \ldots, b-1\}$ and where we assume that for each fixed $j \in \{1,2, \ldots, s\}$, infinitely many of the $z_{j,i}$, $i \in \mathbb{N}$, are different from $b-1$. This makes the expansion of $z_{j,i}$ unique. If the vector $\boldsymbol{x}_n$ depends on an additional index $n$, then we write $x_{n,j}$ for the components and $x_{n,j,i}$ for the corresponding digits.

Let $u, v$ be two non-negative real numbers. Assume that the base $b$ expansions are given by $u = u_r b^r + u_{r-1} b^{r-1} + \cdots$ and $v = v_s b^s + v_{s-1} b^{s-1} + \cdots$ for some $r, s \in \mathbb{Z}$ and $u_r, u_{r-1}, \ldots, v_s, v_{s-1}, \ldots \in \{0, 1, \ldots, b-1\}$, where again we assume that infinitely many of the $u_i$ and also infinitely many of the $v_i$ are different from $b-1$, which makes the expansions unique. In the following we set $u_{r+1} = u_{r+2} = \cdots = 0$ and analogously $v_{s+1} = v_{s+2} = \cdots = 0$. 

We introduce the digit-wise addition $\oplus_b$ and subtraction $\ominus_b$ modulo $b$. We have $z = u \oplus_b v$, if $z$ has base $b$ expansion $z = z_{\max\{ r,s \}} b^{\max \{ r,s \}} + z_{\max \{r,s \} -1} b^{\max \{ r,s \}-1} + \cdots$, where $z_i = u_i + v_i \pmod{b}$ for all $i = \max \{ r,s\}, \max \{ r,s \}-1, \max \{r,s\} -2, \ldots$. Similarly we define $z = u \ominus_b v$ by defining the digits by $z_i = u_i - v_i \pmod{b}$ for $i = \max \{r,s\}, \max \{r,s \} -1, \max \{r,s \}-2, \ldots$.

For two vectors $\boldsymbol{x} = (x_1, \ldots, x_s)^\top \in \mathbb{R}^s$ and $\boldsymbol{y} = (y_1, \ldots, y_s)^\top \in \mathbb{R}^s$ we write $\boldsymbol{x} \le \boldsymbol{y}$ if $x_j \le y_j$ for all $1 \le j \le s$. For a subset $u \subseteq \{1, 2, \ldots, s\}$, the vector $(\bx_u, \boldsymbol{1}_{-u})$ denotes the vector whose $j$th component is $x_j$ is $j \in u$ and $1$ if $j \notin u$.

\subsection{Elementary Intervals}

We define elementary intervals in base $b$, where $b \ge 2$ is an integer, in the following way: Let $\boldsymbol{d} = (d_1, \ldots, d_s)^\top$ be an integer vector, and let $\nu \ge 0$ be an integer. We assume that $d_j \ge 0$ and that $| \boldsymbol{d} | = d_1 + d_2 + \dots + d_s = \nu$. Define the set $$ K_{\boldsymbol{d}} = \{\boldsymbol{a} = (a_1, \ldots, a_s)^\top \in \mathbb{N}_0^s: a_j < b^{d_j} \}.$$ Then for a given $\boldsymbol{d}$ and a vector $\boldsymbol{a} \in K_{\boldsymbol{d}}$ the elementary interval $I_{\boldsymbol{a}, \boldsymbol{d}}$ is given by
\begin{equation*}
I_{\boldsymbol{a}, \boldsymbol{d}} = \prod_{j=1}^s \left[ \frac{a_j}{b^{d_j}}, \frac{a_j+1}{b^{d_j}} \right).
\end{equation*}
Obviously we have $\mathrm{Vol}(I_{\boldsymbol{a}, \boldsymbol{d}}) = b^{-|\boldsymbol{d}| } = b^{-\nu}$ and for a given  $\boldsymbol{d}$, the elementary intervals $\mathcal{I}_{\boldsymbol{d}} = \{I_{\boldsymbol{a}, \boldsymbol{d}}: \boldsymbol{a} \in K_{\boldsymbol{d}} \} $ partition the unit cube $[0,1)^s$.

Let $P = \{\boldsymbol{z}_0, \boldsymbol{z}_1, \ldots, \boldsymbol{z}_{L-1} \} \subset [0,1)^s$ be a chosen point set which we use to represent the data points $\XX$. %We first derive weights for a particular partitioning  $\mathcal{I}_{\boldsymbol{d}}$ of the unit cube, which is fixed by fixing a vector $\boldsymbol{d}$. To shorten the notation, 
We set $\XX_{\boldsymbol{a}, \boldsymbol{d}} = \XX \cap I_{\boldsymbol{a}, \boldsymbol{d}}$ and $P_{\boldsymbol{a}, \boldsymbol{d}} = P \cap I_{\boldsymbol{a}, \boldsymbol{d}}$. Further let $|\XX_{\boldsymbol{a}, \boldsymbol{d}} |$ and $| P_{\boldsymbol{a}, \boldsymbol{d}}|$ denote the number of elements in these sets. 

For $\boldsymbol{y} \in [0,1)^s$, let $I_{\boldsymbol{d}}(\boldsymbol{y}) = I_{\boldsymbol{a}, \boldsymbol{d}}$ where $\boldsymbol{a} \in K_{\boldsymbol{d}}$ is chosen such that $\boldsymbol{y} \in I_{\boldsymbol{a}, \boldsymbol{d}}$, i.e., $I_{\bd}(\by)$ is the elementary interval which contains $\by$. Further we set $\XX_{\boldsymbol{d}}(\boldsymbol{y}) = \XX \cap I_{\boldsymbol{d}}(\boldsymbol{y})$ and $P_{\boldsymbol{d}}(\boldsymbol{y}) = P \cap I_{\boldsymbol{d}}(\boldsymbol{y})$. For $\nu \in \mathbb{N}_0$ we also define $\XX_\nu(\boldsymbol{y}) = \bigcup_{ \satop{ \boldsymbol{d} \in \mathbb{N}_0^s}{|\boldsymbol{d}| = \nu}} \XX_{\boldsymbol{d}}(\boldsymbol{y})$ and $P_\nu(\boldsymbol{y}) = \bigcup_{ \satop{ \boldsymbol{d} \in \mathbb{N}_0^s}{|\boldsymbol{d}| = \nu}} P_{\boldsymbol{d}}(\boldsymbol{y})$.

Let $\nu \ge 0$ be an arbitrary integer which determines the volume of each elementary interval in the partitions, which is $b^{-\nu}$. Further let
\begin{equation*}
K_{\nu} = \bigcup_{\satop{\boldsymbol{d} \in \mathbb{N}_0^s}{ |\boldsymbol{d} | = \nu}} K_{\boldsymbol{d}}.
\end{equation*}

{
We will use the well-known combination principle (see~\cite{sparsegrids}) in order to work with the set $K_\nu$. Since we use the principle with indicator functions instead of interpolation operators, we recall its proof below.
\begin{lemma}\label{lem:inex}
There holds the combination principle for indicator functions $\boldsymbol{1}_{\cdot}$
\begin{equation}\label{inexform}
\boldsymbol{1}_{\boldsymbol{a} \in K_\nu} = \sum_{q=0}^{s-1} (-1)^q {s-1 \choose q} \sum_{\satop{\boldsymbol{d} \in \mathbb{N}_0^s}{|\boldsymbol{d} | = \nu-q}} \boldsymbol{1}_{\boldsymbol{a} \in K_{\boldsymbol{d}}}, \quad \mbox{for all } \boldsymbol{a} \in \mathbb{N}_0^s,
\end{equation}
where for $\nu-q < 0$ the sum $\sum_{\satop{\boldsymbol{d} \in \mathbb{N}_0^s}{|\boldsymbol{d} | = \nu-q}} \boldsymbol{1}_{\boldsymbol{a} \in K_{\boldsymbol{d}}}$ is set to $0$.
\end{lemma}
\begin{proof}
Assume $\nu\in\N$ and $\boldsymbol{a}\in \N_0^s$ fixed and note that we use the convention $\binom{n}{k}=0$ if $n<k$ or $k<0$. If the left-hand side of~\eqref{inexform} is zero, also the right-hand side is zero by definition. Hence, we assume $\boldsymbol{a}\in K_\nu$. By construction of $K_{\boldsymbol{d}}$, there exists a minimal multi-index $\boldsymbol{d}_0$ such that $\boldsymbol{a}\in K_{\boldsymbol{d}_0}$ and each $\boldsymbol{d}\in\N_0^s$ with $\boldsymbol{a}\in K_{\boldsymbol{d}}$ satisfies $\boldsymbol{d}_0\leq \boldsymbol{d}$ entry-wise. The number of $\boldsymbol{d} \ge \boldsymbol{d}_0$ with $|\boldsymbol{d}| = \nu-q$ is equal to the number of $\boldsymbol{e} \ge \boldsymbol{0}$ with $|\boldsymbol{e}| = \nu-q-|\boldsymbol{d}_0|$. This number is given by $\binom{\nu-q-|\boldsymbol{d}_0| + s-1}{\nu-q-|\boldsymbol{d}_0|}$.
%
%
%
%If $\nu-q \geq |\boldsymbol{d}_0|$, we find all such $\boldsymbol{d}$ with $|\boldsymbol{d}|=\nu-q$ by $(\nu-q-|\boldsymbol{d}_0|)$-times  increasing an arbitrary entry of $\boldsymbol{d}_0$ by one. There are $\binom{\nu-q-|\boldsymbol{d}_0| + s-1}{\nu-q-|\boldsymbol{d}_0|}$ possibilities to do that. 
Hence, the identity~\eqref{inexform} 
simplifies to
\begin{align}\label{eq:inex}
    1=\sum_{q=0}^{s-1} (-1)^q {s-1 \choose q} \binom{r-q + s-1}{r-q}
\end{align}
with $r:=\nu-|\boldsymbol{d}_0|$.
This can be shown in a straightforward fashion by induction on $s$. For $s=1$,~\eqref{eq:inex} is obviously true.
Assume~\eqref{eq:inex} holds for some $s\in\N$. Then, we use the recursive identity $\binom{r}{k}=\binom{r-1}{k-1}+\binom{r-1}{k}$ to obtain
\begin{align}\label{eq:inex1}
\begin{split}
 \sum_{q=0}^{s} (-1)^q {s \choose q} \binom{r-q + s}{r-q}   =
 \sum_{q=0}^{s} &(-1)^q\left( {s-1 \choose q-1} \binom{r-q + s-1}{r-q-1}+
  {s-1 \choose q} \binom{r-q + s-1}{r-q-1}\right.\\
  &+\left.
 {s-1 \choose q-1} \binom{r-q + s-1}{r-q}+
 {s-1 \choose q} \binom{r-q + s-1}{r-q}\right).
 \end{split}
\end{align}
The sum over the last term equals one by use of the induction assumption. It remains to show that the other terms cancel each other. An index shift in $q$ in the second term shows
\begin{align*}
    \sum_{q=0}^{s} &(-1)^q\left( {s-1 \choose q-1} \binom{r-q + s-1}{r-q-1}+
  {s-1 \choose q} \binom{r-q + s-1}{r-q-1}\right) \\
  &= 
   \sum_{q=0}^{s} (-1)^q {s-1 \choose q-1} \left(\binom{r-q + s-1}{r-q-1}
 - \binom{r-q + s}{r-q}\right)= -\sum_{q=0}^{s} (-1)^q {s-1 \choose q-1}\binom{r-q + s-1}{r-q},
\end{align*}
which equals the negative third term in~\eqref{eq:inex1}. This concludes the induction and hence the proof.
\end{proof}
}

\subsection{Derivation of the weights $\{ W_{\XX, P, \nu, \ell} \}_{\ell=0}^{L-1} $}

First consider the case of a fixed partition determined by a vector $\boldsymbol{d}$
\begin{align*}
\frac{1}{N} \sum_{n=1}^N f^2_\theta(\bx_n) = & \sum_{ \satop{ \boldsymbol{a} \in K_{\boldsymbol{d}} }{  \XX_{\boldsymbol{a}, \boldsymbol{d}} \neq \emptyset} } \frac{ |\XX_{\boldsymbol{a}, \boldsymbol{d} }| }{N} \frac{1}{| \XX_{\boldsymbol{a}, \boldsymbol{d}} | } \sum_{\satop{n=1}{ \bx_n \in I_{\boldsymbol{a}, \boldsymbol{d}} }}^N f^2_\theta(\bx_n) \approx  \sum_{ \satop{ \boldsymbol{a} \in K_{\boldsymbol{d}} }{   P_{\boldsymbol{a}, \boldsymbol{d}} \neq \emptyset} } \frac{ | \XX_{\boldsymbol{a}, \boldsymbol{d} } | }{N} \frac{1}{ | P_{\boldsymbol{a}, \boldsymbol{d}} | } \sum_{\satop{\ell=0}{ \bz_\ell \in I_{\boldsymbol{a}, \boldsymbol{d}} }}^{L-1} f^2_\theta(\bz_\ell) \\ = & \sum_{\ell=0 }^{L-1} f^2_\theta(\bz_\ell)  \sum_{ \satop{ \boldsymbol{a} \in K_{\boldsymbol{d}} }{ \bz_\ell \in I_{\boldsymbol{a}, \boldsymbol{d}} } } \frac{ | \XX_{\boldsymbol{a}, \boldsymbol{d} } | }{N} \frac{1}{ | P_{\boldsymbol{a}, \boldsymbol{d}} | }  =  \sum_{\ell=0}^{L-1} f^2_\theta(\boldsymbol{z}_\ell) \frac{|\XX_{\boldsymbol{d}}(\boldsymbol{z}_\ell)| }{N} \frac{1}{| P_{\boldsymbol{d}}(\boldsymbol{z}_\ell)|}.
\end{align*}

To obtain an analogous formula which incorporates all possible partitions, we can proceed in the following way. Using the inclusion-exclusion formula \eqref{inexform}, we obtain the approximation
\begin{align}\label{app_X}
\frac{1}{N} \sum_{n=1}^N f_\theta^2(\bx_n) \approx & \sum_{\ell=0}^{L-1} f_\theta^2(\bz_\ell) \underbrace{ \sum_{q=0}^{s-1} (-1)^q { s-1 \choose q} \sum_{\satop{ \boldsymbol{d} \in \mathbb{N}_0^s}{ |\boldsymbol{d}| = \nu-q}}   \frac{ | \XX_{ \boldsymbol{d} }(\boldsymbol{z}_\ell) | }{N} \frac{1}{ | P_{\boldsymbol{d}}( \boldsymbol{z}_\ell) | }}_{= W_{\XX, P, \nu, \ell} }.
\end{align}
While the formula seems quite expensive to compute, we will present an efficient algorithm to do this in Section~\ref{sec:effalg}, below.

% Conceivably, one could also average over all partitions instead of using the inclusion-exclusion formula to obtain an estimation of the sum. However, this does not yield good theoretical approximation properties when we prove the error bounds below and also numerical experiments do not support this approach.

\subsection{Derivation of the weights $\{ W_{\XX, \YY, P, \nu, \ell} \}_{\ell=0}^{L-1}$}

The second set of weights can be derived in a similar manner. For a given $\boldsymbol{d} \in \mathbb{N}_0^s$ (i.e., partition) we use the estimation
\begin{align*}
\frac{1}{N} &\sum_{n=1}^N y_n f_\theta(\bx_n) =   \sum_{\satop{\boldsymbol{a} \in K_{\boldsymbol{d}}}{\XX_{\boldsymbol{a}, \boldsymbol{d}}  \neq \emptyset }}  \frac{ 1  }{N}  \sum_{ \satop{n=1}{\bx_n \in I_{\boldsymbol{a}, \boldsymbol{d}}} }^N y_n f_\theta(\bx_n) \approx  \sum_{\satop{\boldsymbol{a} \in K_{\boldsymbol{d}}}{\XX_{\boldsymbol{a}, \boldsymbol{d}}  \neq \emptyset }}  \frac{ 1  }{N}  \sum_{ \satop{n=1}{\bx_n \in I_{\boldsymbol{a}, \boldsymbol{d}}} }^N \frac{y_n}{|\XX_{\ba, \bd}|} \sum_{m=1 \atop \bx_m \in I_{\ba, \bd}}^N f_\theta(\bx_m) \\ &\approx  \sum_{\satop{\boldsymbol{a} \in K_{\boldsymbol{d}}}{ P_{\boldsymbol{a}, \boldsymbol{d}}  \neq \emptyset }}  \frac{ 1} {N} \sum_{\satop{n=1}{\bx_n \in I_{\boldsymbol{a}, \boldsymbol{d}} }}^N y_n \frac{1}{  | P_{\boldsymbol{a}, \boldsymbol{d}} |  }  \sum_{ \satop{\ell=0}{\bz_\ell \in I_{\boldsymbol{a}, \boldsymbol{d}}} }^{L-1}  f_\theta(\bz_\ell)   =   \sum_{ \ell=0 }^{L-1}  f_\theta(\bz_\ell) \sum_{\satop{\boldsymbol{a} \in K_{\boldsymbol{d}}}{ \bz_\ell \in I_{\boldsymbol{a}, \boldsymbol{d}}  }}  \frac{ 1 }{N} \frac{1}{  | P_{\boldsymbol{a}, \boldsymbol{d}} | }   \sum_{\satop{n=1}{\bx_n \in I_{\boldsymbol{a}, \boldsymbol{d}} }}^N y_n \\ &=  \sum_{ \ell=0 }^{L-1}  f_\theta(\bz_\ell)  \frac{ 1 }{N} \frac{1}{  | P_{\boldsymbol{d}}(\bz_\ell) | }   \sum_{\satop{n=1}{\bx_n \in I_{ \boldsymbol{d}}(\bz_\ell) }}^N y_n. 
\end{align*}

We use again the inclusion-exclusion formula \eqref{inexform} to obtain
\begin{align*}
\frac{1}{N} \sum_{n=1}^N y_n f_\theta(\bx_n)  \approx  & \sum_{ \ell=0 }^{L-1}  f_\theta(\bz_\ell) \underbrace{ \sum_{q=0}^{s-1} (-1)^q {s-1 \choose q} \sum_{\satop{\boldsymbol{d} \in \mathbb{N}_0^s}{| \boldsymbol{d}| = \nu-q}}   \frac{ 1 }{N} \frac{1}{  | P_{\boldsymbol{d}}(\bz_\ell) | }   \sum_{\satop{n=1}{\bx_n \in I_{\boldsymbol{d}}(\bz_\ell) }}^N y_n }_{= W_{\XX, \YY, P, \nu, \ell} }.
\end{align*}

\section{Digital nets}

Let $\alpha \ge 1$ be an integer. As point set $P$ we use an order $\alpha$ digital $(t,m,s)$-net in base $b$. If $\alpha = 1$ we call those point sets a digital $(t,m,s)$-net in base $b$. The point sets are designed to be well distributed in the unit cube $[0,1)^s$. We first introduce order $1$   $(t,m,s)$-nets (which we simply call $(t,m,s)$-nets).

\subsubsection*{$(t,m,s)$-nets}

A point set $P \subset [0,1)^s$  consisting of $L=b^m$ elements is called a $(t,m,s)$-net if every elementary interval $I_{\boldsymbol{a}, \boldsymbol{d}}$ with $|\boldsymbol{d}| = m-t$ contains exactly $b^t$ points, i.e., if $|P_{\boldsymbol{a}, \boldsymbol{d}}| = b^t$ for all $\boldsymbol{a} \in K_{\boldsymbol{d}}$ and all $\bd \in \mathbb{N}_0$ with $|\bd| = m-t$. As a consequence, if $P$ is a $(t,m,s)$-net we have for any $\nu \le m-t$ and $\boldsymbol{d} \in \mathbb{N}_0^s$ with $|\boldsymbol{d}| = \nu-q$, $0 \le q \le \min\{s-1, \nu\}$ that
\begin{equation*}
|P_{\boldsymbol{a}, \boldsymbol{d}} | = b^{m - \nu + q},
\end{equation*}
see Figure~\ref{fig:dignet1} for an example with $b=m=2$.

\begin{figure}
\includegraphics[trim = {0 300 0 50}, clip,  width=0.35\textwidth]{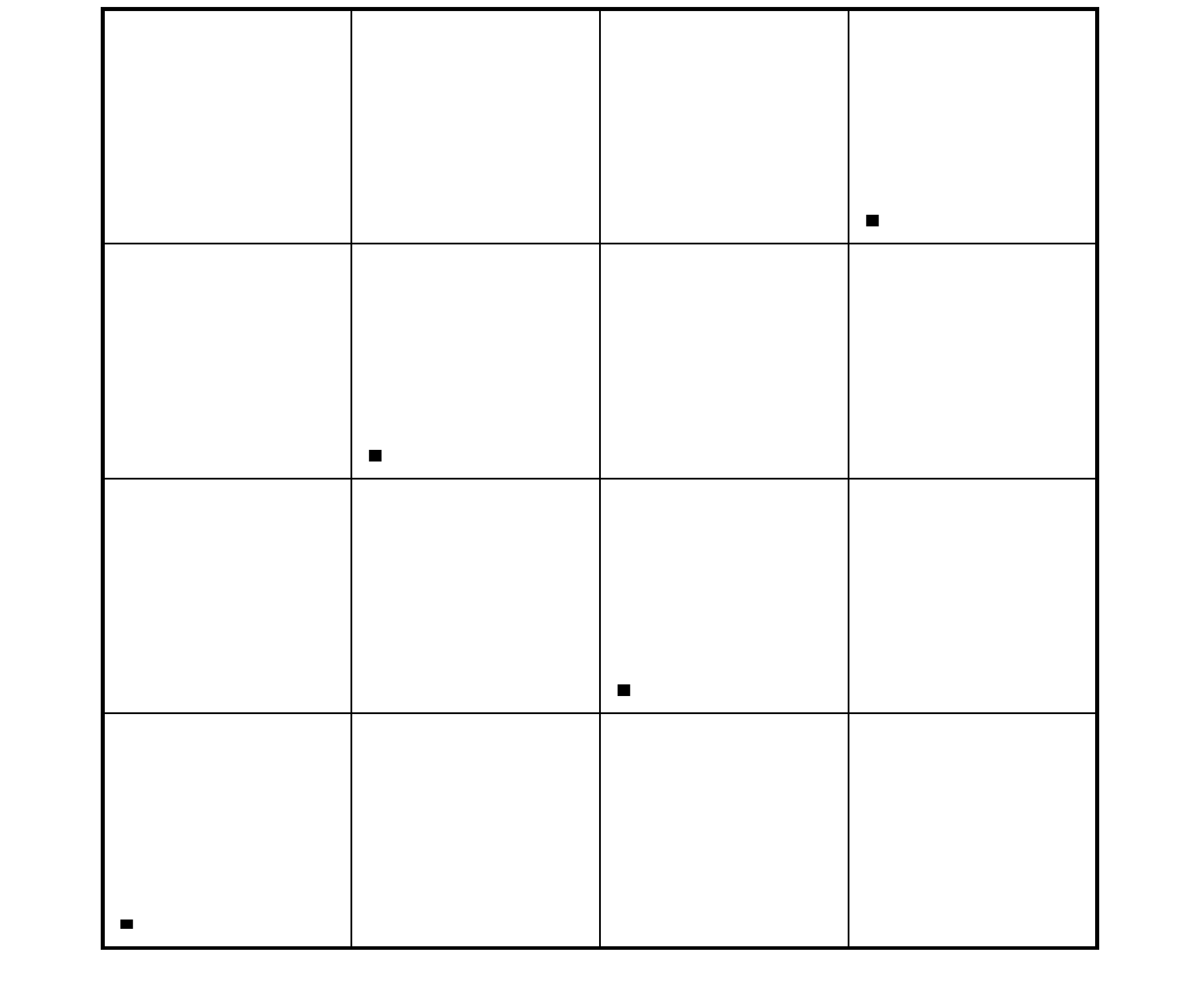}%
\caption{A $(0,2,2)$-net in base $2$. The square can be partitioned into $4$ columns such that each column contains exactly one point. It can also be partitioned into $4$ rows such that each row contains exactly one point, and it can be partitioned into $4$ squares such that each square contains exactly one point.}
\label{fig:dignet1}
\end{figure}

Under certain constraints on the value of $t$, it is known how to construct such $(t,m,s)$-nets. Roughly speaking, $t$ can be chosen independently of $m$, but depends at least linearly on $s$, see \cite{DP10} for more information. Explicit constructions of $(t,m,s)$-nets are known, with the first examples due to Sobol\cprime{} \cite{Sob} and Faure \cite{F82}, before Niederreiter \cite{N87} introduced the general digital construction principle, which we describe in the following.

\subsubsection*{Digital nets}

Let $b$ be a prime number and let $\mathbb{Z}_b = \{0, 1, \ldots, b-1\}$ be the finite field with $b$ elements where we identify the elements with the corresponding integers, but with addition and multiplication carried out modulo $b$. Let $C_1, \ldots, C_s \in \mathbb{Z}_b^{m \times m}$ be $s$ matrices, which determine the digital net (explicit constructions of such matrices are due to Sobol\cprime \cite{Sob}, Faure \cite{F82}, Niederreiter \cite{N87} and others, see \cite{DP10}). 

In the following we describe how to construct the $j$th component $z_{\ell,j}$ of the $\ell$th point $\bz_\ell$ of the digital net. The digital net is given by $\boldsymbol{z}_\ell = (z_{\ell,1}, z_{\ell,2}, \ldots, z_{\ell,s})$ for $\ell = 0, 1, \ldots, b^{m}-1$. 

Let $\ell \in \{0, 1, \ldots, b^{m} - 1\}$. We represent $\ell$ in its base $b$ expansion $\ell = \ell_0 +  \ell_1 b + \cdots +  \ell_{m-1} b^{m-1} $, where $\ell_0, \ell_1, \ldots, \ell_{m-1} \in \mathbb{Z}_b$. We define the corresponding vector $\vec{\ell} = (\ell_0, \ell_1, \ldots, \ell_{m-1})^\top \in \mathbb{Z}_b^{m}$. Let $\vec{z}_{\ell,j} = (z_{\ell,j,1}, z_{\ell,j,2}, \ldots, z_{\ell,j,m})^\top$ be given by $$\vec{z}_{\ell,j} = C_j \vec{\ell}.$$ Then $$z_{\ell,j} = \frac{z_{\ell,j,1}}{b} + \frac{z_{\ell,j,2}}{b^2} + \cdots + \frac{z_{\ell,j,m}}{b^{m}}.$$

A digital net which satisfies the $(t,m,s)$-net property is called a digital $(t,m,s)$-net.

\subsubsection*{Higher order digital nets}

The theory of higher order digital nets was introduced in \cite{D08}. Rather than giving an introduction to this topic in detail, we briefly describe how to construct higher order digital nets from existing digital nets. For more details we refer to \cite[Chapter~15]{DP10}.

Let $\alpha \ge 2$ be an integer. We introduce the digit interlacing function $D_\alpha$ in the following. Let $z_1, z_2, \ldots, z_\alpha \in [0,1)$ with base $b$ expansions $z_j = z_{j,1} b^{-1} + z_{j,2} b^{-2} + \cdots$, with $z_{j,i} \in \{0, 1, \ldots, b-1\}$. Then $D_\alpha: [0,1)^\alpha \to [0,1)$ is given by
\begin{equation*}
D_\alpha(z_1, \ldots, z_\alpha) = \sum_{i=1}^\infty \sum_{j=1}^\alpha \frac{z_{j,i}}{b^{(i-1)\alpha + j}}.
\end{equation*}
For vectors $\bz \in [0,1)^{\alpha s}$ we set $$D_\alpha(\bz) = (D_\alpha(z_1, \ldots, z_\alpha), D_\alpha(z_{\alpha+1}, \ldots, z_{2 \alpha}), \ldots, D_\alpha(z_{(s-1)\alpha+1}, \ldots, z_{\alpha s})).$$

We construct an order $\alpha$ digital $(t,m,s)$-net as follows. Let $\{\bz_0, \bz_1, \ldots, \bz_{b^{m}-1}\} \subset [0,1)^{\alpha s}$ be a digital $(t,m,\alpha s)$-net. Then the point set
\begin{equation}\label{honet}
\{D_\alpha(\bz_0), D_\alpha(\bz_1), \ldots, D_\alpha(\bz_{b^{m}-1}) \} \subset [0,1)^s
\end{equation}
is an order $\alpha$ digital $(t_\alpha, m, s)$-net. This is the so-called digit interlacing construction of higher order digital nets introduced in \cite{D08}. The $t_\alpha$ values satisfy the bound (see \cite[Lemma~15.6]{DP10})
\begin{equation*}
t_\alpha \le \alpha t + \alpha \lfloor s (\alpha-1) /2 \rfloor.
\end{equation*}
We notice that the point set \eqref{honet} is also a $(t_\alpha, m, s)$-net, i.e., every elementary interval $I_{\ba,\bd}$ with $|\bd| = m-t_\alpha$ has exactly $b^{t_\alpha}$ points (see \cite[Proposition~15.8]{DP10}).

\subsection*{The weights when $P$ is a $(t,m,s)$ net}
The weights are quite costly to compute with the most expensive part being the counting {number of data points $|\XX_\bd(\bz_\ell)|$ in all elementary intervals}
 for $|\bd|=\nu-q$. Since there are $\binom{s-1+\nu-q}{s-1}$ different vectors $\bd$ with $|\bd|=\nu-q$,
the cost increases exponentially with the dimension.
However, the calculation of the weights simplifies dramatically when $P$ is a $(t,m,s)$-net and $m-t \ge \nu \in \mathbb{N}$. {In this case, there holds $|P_{\boldsymbol{d}}(\boldsymbol{z}_\ell)| = b^{m-\nu+q}$ for $|\boldsymbol{d}|=\nu-q$ and formula for the weights simplifies to}
\begin{align}\label{dignet_form}
W_{\XX, P, \nu, \ell} = & \frac{b^{\nu-m}}{N} \sum_{q=0}^{s-1} (-1)^{-q} {s-1 \choose q} \frac{1}{b^q} \sum_{\satop{ \boldsymbol{d} \in \mathbb{N}_0^s}{|\boldsymbol{d}| = \nu-q}} |\XX_{\boldsymbol{d}}(\boldsymbol{z}_\ell) |. %= \frac{b^{\nu-m}}{N} \left| \XX_\nu(\boldsymbol{z}_\ell) \right|.
\end{align}
Similarly, we also have
\begin{align}\label{dignet_form_y}
W_{\XX, \YY, P, \nu, \ell} = & \frac{b^{\nu-m}}{N} \sum_{q=0}^{s-1} (-1)^q {s-1 \choose q} \frac{1}{b^q} \sum_{\satop{\boldsymbol{d} \in \mathbb{N}_0^s}{| \boldsymbol{d}| = \nu-q}} \sum_{\satop{n=1}{\bx_n \in I_{\boldsymbol{d}}(\boldsymbol{z}_\ell) }}^N y_n. 
\end{align}
Figure~\ref{fig:dignet2} illustrates the fact that the weight-formula essentially computes the ratio of data points in $\XX$ and points in $P$ for given elementary intervals.

\begin{figure}
\includegraphics[trim = {0 300 0 50}, clip, width=0.35\textwidth]{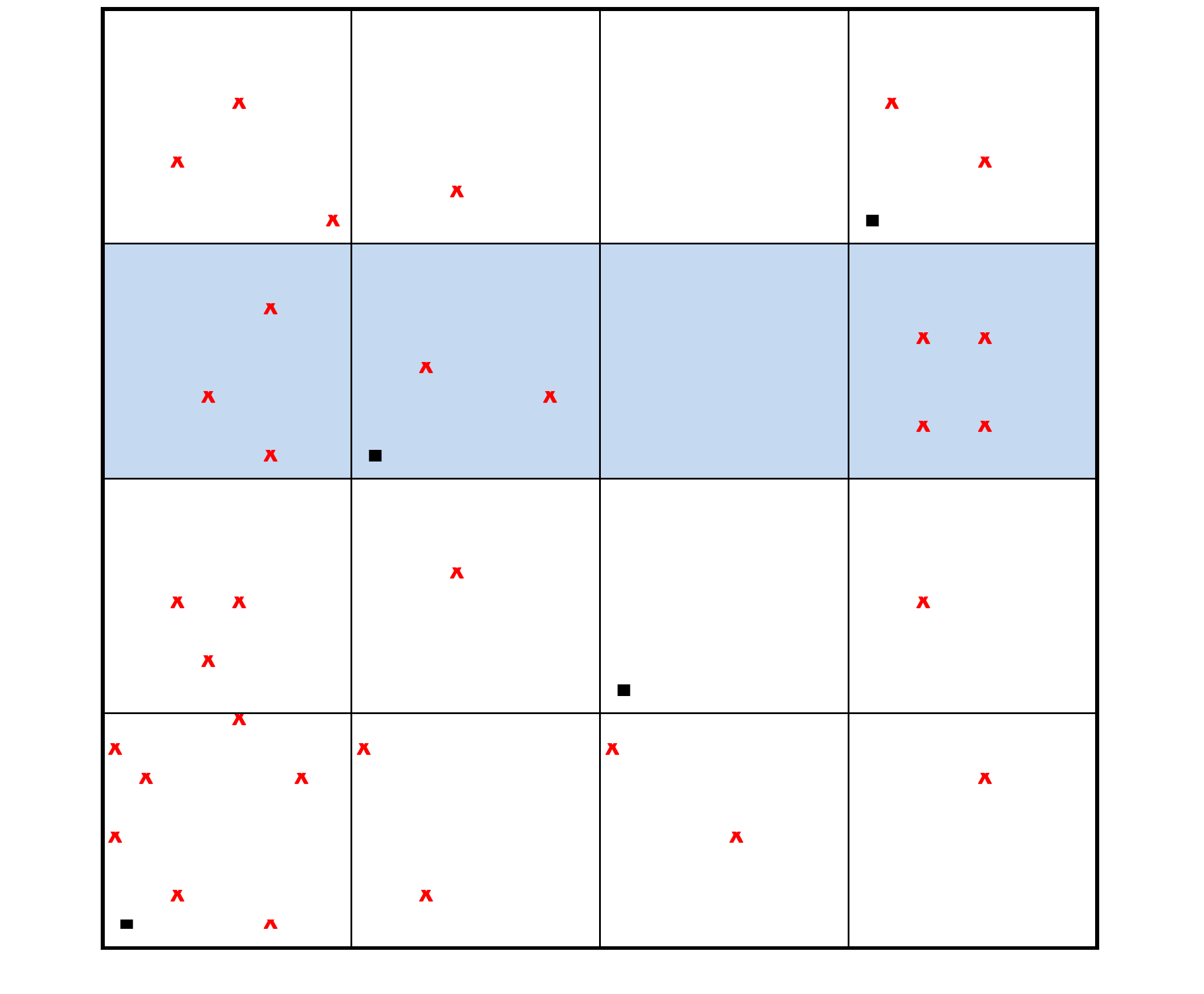}%
\caption{A $(0,2,2)$-net in base $2$ and data points (in red). For a given point of the $(0,2,2)$-net we calculate the proportion of data points in the elementary intervals and use the inclusion-exclusion principle to obtain the weight.}
\label{fig:dignet2}
\end{figure}

\section{Efficient computation of the weights}\label{sec:effalg}

The definition of $W_{\XX, P,\nu, \ell}$ and $W_{\XX, \YY, P,\nu,\ell}$ requires one to compute the 
values $\XX_{\nu}(\boldsymbol{z}_\ell)$ for all points $\boldsymbol{z}_\ell$ in $P$. Since these values are 
derived from the data set and therefore the computational cost depends on $N$, we need an efficient method of computing them. We exploit the fact that we use digital nets for $P$. %While an efficient method for general point sets would be highly desirable, we did not succeed in finding one.

\subsection{Efficient computation of $W_{\XX, P, \nu, \ell}$}

The hardest part in computing the weights $W_{\XX, P, \nu, \ell}$ is the computation of
\begin{equation}\label{eq:defS}
S_{\nu-q}(\boldsymbol{z}_\ell) = \sum_{\boldsymbol{d} \in \mathbb{N}_0^s \atop |\boldsymbol{d}| = \nu-q} |\XX_{\boldsymbol{d}}(\boldsymbol{z}_\ell)|, \quad 0 \le q \le \min\{\nu, s-1\},
\end{equation}
which has to be computed for all $\ell = 0, 1, \ldots, b^{m}-1$. If $\nu-q = 0$ then we have $S_0(\bz_\ell) = |\XX| = N$, and so this case is straightforward. We consider now $\nu-q > 0$.

The idea is the following: In the first step, for a given data point $\bx_n$ we find the smallest elementary interval $I$ with side length of each side at least $b^{-\nu+q}$, which contains the data point and the point $\boldsymbol{z}_\ell$. Given that, for this data point, we count the number of all $\bd$ with $|\bd| = \nu-q$ such that the elementary interval $I \subseteq I_{\bd}(\bz_\ell)$. By doing this for all data points, we obtain the required values. We state the algorithm for a generic point $\bz$.

\begin{algorithm}\label{alg:S}
 \textbf{Input: }$\bz\in [0,1)^s$, $\XX$, and $r = \nu - q \geq 1$.\\
 Set $S_r(\bz)=0$. 
 
 For $n=1,\ldots, N$ do:
 \begin{enumerate}
  \item For $j=1,\ldots,s$ do:
  \begin{enumerate}
  \item[] Find the maximal $i_j \in\{0,\ldots, r\}$ such that
  the first $i_j$ digits of $z_j$ and $x_{n,j}$ are the same, i.e.,
  \begin{align*}
   z_{j,1} = x_{n,j,1}, \ldots, z_{j,i_j} = x_{n,j,i_j}.
  \end{align*}
  \end{enumerate}
  End For
  \item Set $\bi = (i_{1}, i_2, \ldots, i_s)$.
  \item Set $S_r(\bz)=S_r(\bz)+ \#\{\boldsymbol{d} \in \N_0^s: |\boldsymbol{d}|=r, \boldsymbol{d} \leq \bi \}$.
 \end{enumerate}
\quad End For

\noindent
\textbf{Output: } $S_r(\bz)$ %$S(\bz)= \sum_{\bm\in \N_0^s\atop|\bm|=t} \sum_{\ba\in K_\bm-1 \atop \bz\in I_{\ba,\bm}}\#(\XX\cap I_{\ba,\bm})$$
\end{algorithm}
We show below in Lemma~\ref{lem:algSN} that the numbers $S_r(\bz)$ computed by Algorithm~\ref{alg:S} are indeed equal to  $\sum_{\boldsymbol{d} \in \N_0^s\atop|\boldsymbol{d} |=r} | \XX_{\bd}(\bz)|$.

In Algorithm~\ref{alg:S}, one has to calculate the numbers $$N_{r,\bi}:=\#\set{\bd \in \N_0^s}{|\bd|=r, \bd \leq \bi}$$ in an efficient way. To that end, we propose Algorithm~\ref{alg:N} below. The idea of this algorithm is to update the numbers $N_{r,\boldsymbol{i}}$ coordinate by coordinate. Explicitly stated, let $N_{j, r,\boldsymbol{i} }$ be the number of vectors $(d_1, \ldots, d_{j-1}) \in \mathbb{N}_0^{j-1}$ such that $0 \le d_v \le i_v$ for $1 \le v \le j-1$ and $d_1 + \cdots + d_{j-1} = r$. Then we have the recursive identity
 \begin{align*}
  N_{j,r,\boldsymbol{i} }=\sum_{w=\max\{r-i_j,0\}}^r N_{j-1,w,\boldsymbol{i} },
 \end{align*}
 since for each $(d_1, \ldots, d_{j-1})$ with $\max\{r-i_j, 0\} \le d_1 + \cdots + d_{j-1} \le r$, we can add a new coordinate $0 \le d_j \le i_j$ such that $d_1 + \cdots + d_{j-1} + d_j = r$.
 
 Since we only need the result in dimension $s$ but not the intermediate dimensions, we can overwrite the numbers from the previous dimension in each iteration.

\begin{algorithm}\label{alg:N}
 \textbf{Input: }$\boldsymbol{i} \in\N_0^s$ and $r\geq 1$.\\
 \begin{enumerate}
  \item For $r' = 0, 1, \ldots, r$, set
  \begin{equation*}
       N_{r', \boldsymbol{i}} = \begin{cases} 1 & \mbox{if } r' \le i_1, \\ 0 & \mbox{if } r' > i_1. \end{cases}
  \end{equation*}
  End For
   \item For ${j}=2,\ldots,s$ do:
 \begin{enumerate}
 \item[] For $r'=0,\ldots, r$, set 
 \begin{align*}
  N_{r', \boldsymbol{i}} = \sum_{w =  \max\{r'-i_{j},0\}}^{r'} N_{w, \boldsymbol{i} }.
 \end{align*}
 End For
 \end{enumerate}
 End For
 \end{enumerate}
\textbf{Output:} $N_{r',\boldsymbol{i} } $
\end{algorithm}

%As we have seen before, Algorithm~\ref{alg:N} computes the numbers $$N_{r',\bi} =\#\{ \bd \in \N_0^s:  |\bd |=r, \bd \leq \bi \} \mbox{ for all } 0\leq r'\leq r.$$

\begin{lemma}\label{lem:algSN}
Algorithm~\ref{alg:S} computes the correct value $S_r(\bz)$. Further, the cost of Algorithm~\ref{alg:N} is $\mathcal{O}(rs)$ and the cost of Algorithm~\ref{alg:S} is $\mathcal{O}(rsN)$.
\end{lemma}

\begin{proof}
To show the correctness of Algorithm~\ref{alg:S}, first notice that
 \begin{align*}
 S_{r}(\bz) = \sum_{\bd \in \N_0^s \atop |\bd|=r } |\XX_{\bd}(\bz)|
= \sum_{n=1}^N  \sum_{\bd \in \N_0^s \atop |\bd|=r}   \b{1}_{\bx_n  \in I_{\bd}(\bz) } = \sum_{n=1}^N \sum_{\bd \in \N_0^s \atop |\bd| = r} \sum_{\ba \in K_{\bd}} \b{1}_{\bx_n, \bz \in I_{\bd, \ba}},
\end{align*}
where $\b{1}_{(\cdot)}$ denotes the indicator function. Thus, for each $\bx_n \in \XX$, we need to count the number of intervals $I_{\ba,\bd}$, $|\bd|=r$, which contain both $\bz$ and $\bx_n$. If $\bz, \bx_n \in I_{\ba,\bd}$, then for each coordinate $j \in \{1, 2, \ldots, s\}$ the first $d_j$ digits of $z_j$ and $x_{n,j}$ have to coincide. In the algorithm, for each coordinate $j$, we first compute the maximum of $i_j \in \{0, 1, \ldots, r\}$ such that at least $i_j$ digits of $z_j$ and $x_{n,j}$ coincide, which implies that $\bx_n \in I_{\bi}(\bz)$. Then for any $\boldsymbol{d} \le \boldsymbol{i}$ we have $\bx_{n} \in I_{\bd}(\bz)$ since $I_{\bd}(\boldsymbol{z}) \supseteq I_{\bi}(\bz)$. Thus $$\sum_{\bd \in \NN_0^s \atop |\bd| = r} \b{1}_{\bx_n \in I_{\bd}(\bz)} = \#\{  \bd \in \mathbb{N}_0^s: |\bd| = r, \bd \le \boldsymbol{i}  \}.$$

Straightforward counting of the steps reveals the statement on the number of operations needed for the algorithms, if one notices that %step~(1a) of Algorithm~\ref{alg:S} can be computed in $\mathcal{O}(\log(r))$ by a binary search (in base $2$ it can be done faster) and that 
$N_{r'}$ in Step~(2) of Algorithm~\ref{alg:N} can be obtained by computing a moving sum of the previous vector. This concludes the proof.
\end{proof}

If the point set $\{\bz_0, \ldots, \bz_{b^m-1}\}$ is a digital net with known upper bound for the $t$-value, Algorithm~\ref{alg:S} can be made even {more efficient by trading the dimension dependence of the constants for a weaker dependence on $L=b^m$.}
\begin{lemma}\label{lem_weight_net}
Let $(\bz_\ell)_{\ell=0,\ldots,b^m-1}$ be a digital $(t,m,s)$-net and let $0 \le r \le m-t$ be an integer. For $\bx_n\in\XX$, let $R_{n,\ell}\in\N$ denote the number of points $\bz_j$, $0\leq j < b^m$ for which $N_{r,\bi}$ in Step~(3) of Algorithm~\ref{alg:S} ({when applied to $\bz=\bz_\ell$}) is non-zero.
Then, Algorithm~\ref{alg:S} can skip all data points $\bx_n\in\XX$ with $R_{n,\ell}\geq \binom{s-1+r}{s-1} b^{m-r}$, using an additional storage of order $N$.
Since $r=\nu-q$ for $q=0,\ldots,s-1$, the cost of computing the weights $W_{\XX,P,\nu,\ell}$, {$\ell=0,\ldots,b^m-1$,} is bounded by
\begin{align*}
\mathcal{O}\Big(N \binom{s-1+\nu}{s-1} b^{m-\nu+s-1}\Big)=\mathcal{O}\Big(N (s-1+\nu)^{s-1} b^{m-\nu+s-1}\Big).
\end{align*}
\end{lemma}
\begin{proof}
Define the function $S_r(\bz,\bx):=\sum_{\bd\in\N_0^s\atop |\bd|=r} \b{1}_{\bx\in I_{\bd}(\bz)}$ such that $S_r(\bz_\ell, \bx_n) = N_{r, \bi}$ in Step (3) of Algorithm~\ref{alg:S}.
By definition of $S_r(\bz)$ in~\eqref{eq:defS}, there holds $S_r(\bz)=\sum_{n=1}^N S_r(\bz,\bx_n)$ as well as
\begin{align*}
 \sum_{\ell=0}^{{b^m-1}}S_r(\bz_\ell,\bx_n) 
 =\sum_{\bd\in\N_0^s\atop |\bd|=r} \sum_{\ell=0}^{{b^m-1}} \b{1}_{\bx_n\in I_{\bd}(\bz_\ell)} \leq \binom{s-1+r}{s-1}b^{m-r}.
\end{align*}
Hence, for fixed $n\in\{1,\ldots,N\}$, we have $S_r(\bz_\ell,\bx_n)>0$ for less than $\binom{s-1+r}{s-1}b^{m-r}$ 
points $\bz_\ell$. %Note that $S_r(\bz_\ell,\bx_n)= N_{r,\bi}$ in Step~(3) of Algorithm~\ref{alg:S}. 
Hence, keeping track of the number of times $N_{r,\bi}>0$ allows us to discard points $\bz_\ell$. 
This concludes the proof.
\end{proof}

\subsection{Efficient computation of $W_{\XX, \YY, P, \nu, \ell}$}

As for $W_{\XX, P, \nu, \ell}$, we also need to efficiently compute the term
\begin{equation*}
T_{\nu - q}(\bz) = \sum_{\satop{\boldsymbol{d} \in \mathbb{N}_0^s}{| \boldsymbol{d}| = \nu - q}}  \sum_{\satop{n=1}{\bx_n \in I_{\boldsymbol{d}}(\boldsymbol{z}) }}^N y_n, \quad 0 \le q \le \min\{\nu, s-1 \},
\end{equation*}
for all $\bz=\bz_\ell$ with $\ell=0,\ldots,b^{m}-1$. For $\nu-q = 0$ we have $T_0(\bz) = \sum_{n=1}^N y_n$. In the following we consider the case $r > 0$. Notice that if $y_n = 1$ for all $n = 1, 2, \ldots, N$, then $S_r(\bz) = T_r(\bz)$.

We propose the following variation of Algorithm~\ref{alg:S}.
\begin{algorithm}\label{alg:S2}
\textbf{Input: }$\bz\in [0,1]^s$, $\XX$, $\YY$, and $r = \nu - q \geq 1$.\\
Set $T_r(\bz)=0$. 
 
For $n=1,\ldots, N$ do:
 \begin{enumerate}
  \item For $j=1,\ldots,s$ do:
  \begin{enumerate}
  \item[] Find the maximal $i_j  \in\{0,\ldots, r\}$ such that the first $i_j$ digits of $z_j$ and $x_{n,j}$ coincide, i.e.,
  \begin{align*}
   z_{j,1} = x_{n,j,1}, \ldots, z_{j,i_j} = x_{n,j, i_j}.
  \end{align*}

  End For
  \end{enumerate}
  \item Set $\bi = (i_1, i_2, \ldots, i_s)$.  
  \item Set $T_r(\bz)= T_r(\bz)+ y_n \#\set{\bd \in \N_0^s}{|\bd|=r, \bd \leq \bi }$.
 \end{enumerate}
 \quad End For

\noindent 
\textbf{Output: } $T_r(\bz)$ %= \sum_{\bm\in \N_0^s\atop|\bm|=\nu-q} \sum_{\ba\in K_\bm-1 \atop \bz\in I_{\ba,\bm}}\sum_{\bx_n\in \XX\cap I_{\ba,\bm}}y_n$
\end{algorithm}

We reuse Algorithm~\ref{alg:N} to compute the numbers $N_{r,\bi} := \#\set{\bd \in \N_0^s}{|\bd|=r, \bd \leq \bi }$.

\begin{lemma}\label{lem:algS2}
Algorithm~\ref{alg:S2} computes the correct values $T_r(\bz)$. The cost of Algorithm~\ref{alg:S2} is $\mathcal{O}(rsN)$.
\end{lemma}
\begin{proof}
 The cost of Algorithm~\ref{alg:S2} is identical to that of Algorithm~\ref{alg:S} and hence follows from Lemma~\ref{lem:algSN}.
 
To show the correctness of Algorithm~\ref{alg:S2}, notice that
 \begin{align*}
 T_{r}(\bz) = &  \sum_{\satop{\boldsymbol{d} \in \mathbb{N}_0^s}{| \boldsymbol{d}|   = r}} \sum_{\satop{n=1}{\bx_n \in I_{\boldsymbol{d}}(\boldsymbol{z}) }}^N y_n =  \sum_{n=1}^N  y_n  \sum_{\bd \in \mathbb{N}_0^s \atop |\bd | = r} \b{1}_{\bx_n \in I_{\bd}(\bz)} = \sum_{n=1}^N y_n N_{r, \bi},
\end{align*}
since in the proof of Lemma~\ref{lem:algSN} we already confirmed that
\begin{align*}
N_{r,\bi}=\sum_{\bd \in \N_0^s\atop|\bd|=r}
\sum_{\ba\in K_\bd }\b{1}_{\bx_n,\bz\in I_{\ba,\bd}}
\end{align*}
and it only remains to add up the $y_n N_{r,\bi}$ for each $\bx_n \in \XX$. This concludes the proof.
\end{proof}
Note that Lemma~\ref{lem_weight_net} applies in the same way for computing the weights $W_{\XX, \YY, P, \nu, \ell}$. We summarize the results of the previous sections in the following theorem.

\begin{theorem}\label{thm:cost2}
The startup cost of computing $\mathrm{app}_L(f_\theta)$ is $\mathcal{O}({s^2} m b^{m} N)$ whereas each 
recomputation with identical data $\XX$ but different $f_\theta$ costs $\mathcal{O}(sb^{m})$. 
{If $P$ is a $(t,m,s)$-net, we can trade a weaker dependence of the startup cost on $L=b^m$ for a stronger dimension dependence in the sense that the cost reduces to
$\mathcal{O}(\min\{s^2m b^m N,N(s-1+\nu)^{s-1}b^{m-\nu+s-1}\})$,} where $\nu\leq m-t$ has to be chosen by the 
user ({we refer to Remark~\ref{rem:cost} below for a detailed discussion of cost vs. error}).
\end{theorem}
\begin{proof}
{The start-up cost and recomputation cost is derived from the definition of the weights in~\eqref{dignet_form}--\eqref{dignet_form_y} together with the cost of computing $\sum_{\bd\in \N_0^s \atop |\bd|=\nu-q} |\XX_\bd(\bz_\ell)|$ in Lemma~\ref{lem:algSN},~\ref{lem_weight_net}, and~\ref{lem:algS2}.}
\end{proof}

\subsection{Updating the weights for new values of $m$ and $\nu$}

We now consider the situation where $\{W_{\XX, P, \nu, \ell} \}_{\ell=0}^{L-1}$ and $\{W_{\XX, \YY, P, \nu, \ell}\}_{\ell=0}^{L-1}$ have already been computed for some given $(t,m,s)$-net $P$ and some given $\nu$, but now one wants to increase the accuracy of the approximation by increasing $m$ and/or $\nu$ to $m' \ge m$ and $\nu' \ge \nu$.

Let $P' = \{\bz_0, \bz_1, \ldots, \bz_{b^{m'}-1}\}$ be a $(t,m',s)$-net and assume that we previously used the first $b^m$ points of $P'$ to calculate the weights $\{ W_{\XX, P, \nu, \ell}\}$ and $\{W_{\XX, \YY, P, \nu, \ell}\}$, i.e., $P = \{ \bz_0, \bz_1, \ldots, \bz_{b^m-1}\}$. If all the previous values of $S_r(\bz_\ell)$ and $T_r(\bz_\ell)$ were stored for $\ell = 0, 1, \ldots, b^m-1$ and $r \in \{\nu, \nu-1, \ldots, \nu-s+1\}$, then only the new values for $r = \nu', \nu'-1, \ldots, \nu+1$ need to be computed. The weights for the remaining points can be computed using Algorithms~\ref{alg:S}, \ref{alg:N}, and \ref{alg:S2}. Hence the cost of computing the weights $\{W_{\XX, P, \nu, \ell}\}_{\ell=0}^{b^m-1}$, $\{W_{\XX, \YY, P, \nu, \ell}\}_{\ell=0}^{b^m-1}$, $\{ W_{\XX, P', \nu, \ell}\}_{\ell=0}^{b^{m'}-1}$, and $\{W_{\XX, \YY, {P'}, \nu, \ell}\}_{\ell = 0}^{b^{m'}-1}$ is of the same order as computing the latter two sets of weights directly, with an additional storage cost for storing $S_r(\bz_\ell)$ and $T_r(\bz_\ell)$, which is of order $b^m \min\{ \nu, s\}$.

\section{Error Analysis}\label{sec:error}

Before analysing the error of our approximation, we derive the formulae for the weights $W_{\XX, P, \nu, \ell}$ and $W_{\XX, \YY, P, \nu, \ell}$ again by less geometrical means, i.e., we are using a Walsh series expansion of the implicit \emph{density} of the data points $\XX$.

\subsection{Derivation of the weights $W_{\XX, P, \nu, \ell}$ and $W_{\XX, \YY, P,\nu, \ell}$ based on Walsh series}\label{ssec_W}

Let $\omega_{k}$ be the $k$-th Walsh function in base $b \ge 2$ for $k \in \mathbb{N}_0^s$ defined by
\begin{align*}
\omega_k(x) = \mathrm{e}^{\tfrac{2\pi i}{b} (k_0x_1+\ldots + k_{j-1}x_j )}\quad\text{where }k=\sum_{i=0}^{j-1} k_i b^i\text{ and } x = \sum_{i=1}^\infty x_ib^{-i},
\end{align*}
for digits $k_i,x_i\in\{0,\ldots,b-1\}$. For $\bx\in [0,1]^s$, we define the multi-dimensional Walsh functions by
\begin{align*}
 \bomega_\bk (\bx):=\prod_{j=1}^s \omega_{k_j}(x_j).
\end{align*}
For details on Walsh functions, see, e.g.,~\cite[Appendix~A]{DP10}.

In the following we derive the approximations for the loss given in \eqref{dignet_form} and \eqref{dignet_form_y} using the Walsh series expansions of the functions $f_\theta$ and $f_\theta^2$. Since both cases are very similar we can treat them at the same time. We use the function $g$ and the coefficients $c_n$, where, in order to derive the weights $W_{\XX, P, \nu, \ell}$ we set
\begin{align} \label{WX}
g = & f_\theta^2, \nonumber \\
c_n = & 1, 
\end{align}
and to derive the weights $W_{\XX, \YY, P, \nu, \ell}$ we set
\begin{align}\label{WXY}
g = & f_\theta, \nonumber \\
c_n = & y_n.
\end{align}

Let $K \subset \mathbb{N}_0^s$ (to be defined later) be a finite subset. Then we define the functions $g_K = \sum_{\bk \in K} \widehat{g}_{\bk} \omega_{\bk}$ and $g_{-K} = \sum_{\bk \in \mathbb{N}^s \setminus K} \widehat{g}_{\bk} \omega_{\bk}$, where $\widehat{g}_{\bk} = \int_{[0,1]^s} g(x) \overline{\omega_{\bk}(x)} \,\mathrm{d} x$ is the $\bk$-th Walsh coefficient of $g$. Hence $g = g_K + g_{-K}$. We choose $K$ such that $g_{-K}$ is 'small' (to be discussed later), so that we can use $g_K$ as an approximation of $g$. Assume that $\|g-g_K\|_\infty < B_K$ for some bound $B_K$ depending on $K$ and $g$ such that $B_K \to 0$ as $|K| \to \infty$.

Then we have
\begin{align}\label{eq:deriv}
\begin{split}
\frac{1}{N} \sum_{n=1}^N c_n g(\bx_n) &=  \frac{1}{N} \sum_{n=1}^N c_n g_K(\bx_n) + \frac{1}{N} \sum_{n=1}^N c_n g_{-K}(\bx_n) \\
&\approx \frac{1}{N} \sum_{n=1}^N c_n g_K(\bx_n)
=  \int_{[0,1]^s} g_K(\bx) \phi_K(\bx)\,d\bx=  \int_{[0,1]^s} g(\bx) \phi_K(\bx)\,d\bx
\end{split}
\end{align}
for the function $$\phi_K(\bx):= \sum_{\bk\in K} \mu_{\bk}\overline{\bomega_\bk(\bx)}$$ with coefficients
\begin{align}\label{eq:defmu}
 \mu_\bk := \frac{1}{N} \sum_{n=1}^N c_n \bomega_\bk(\bx_n).
\end{align}
The last two equalities in~\eqref{eq:deriv} follow immediately from the orthogonality of Walsh-functions, i.e., 
{\begin{align*}
    \int_{[0,1]^s}\bomega_\bk(\bx) \overline{ \bomega_{\bk'}(\bx) } \,d\bx = \begin{cases} 0 &\bk\neq\bk',\\ 1 &\bk=\bk'.\end{cases}
\end{align*}}
The remaining integral is approximated by a $(t,m,s)$-net $\{\bz_0, \bz_1, \ldots, \bz_{b^{m}-1}\}$ in base $b$ with $b^{m}$ points, i.e.,
\begin{align}\label{eq:deriv1}
\frac{1}{N} \sum_{n=1}^N c_n g(\bx_n) \approx \frac{1}{N} \sum_{n=1}^N c_n g_K(\bx_n) =  \int_{[0,1]^s} g(\bx) \phi_K(\bx) \,d \bx \approx \frac{1}{b^{m}} \sum_{\ell=0}^{b^{m} - 1}g(\bz_\ell)\phi_K(\bz_\ell).
\end{align}

We show that the right-hand side above coincides with $\sum_{\ell=0}^{b^{m}-1} f_\theta^2(\bz_\ell) W_{\XX, P, \nu, \ell}$ if we use \eqref{WX} and with $\sum_{\ell=0}^{b^{m}-1} f_\theta(\bz_\ell) W_{\XX, \YY, P, \nu, \ell}$ if we use \eqref{WXY}. To that end, we require the following lemma.

\begin{lemma}\label{lem:indicator}
For any $\bd \in \mathbb{N}_0^s$ there holds 
\begin{equation*}
\phi_{K_\bd}(\bz) = \frac{b^{|\bd|}}{N} \sum_{n=1 \atop \bx_n \in I_{\bd}(\bz) }^N c_n.
\end{equation*}
If $c_n = 1$ for all $n$, then
\begin{equation*}
\phi_{K_\bd}(\bz) = \frac{ b^{|\bd|} }{N}  |\XX_{\bd}(\bz)|.
\end{equation*}
\end{lemma}

\begin{proof}
We first note
\begin{align*}
 \sum_{k=0}^{b^{d}-1}\omega_k(x) =\sum_{k_0=0}^{b-1}\cdots \sum_{k_{d-1}=0}^{b-1} \mathrm{e}^{2\pi i/b (k_0x_1 + k_1x_2 + \ldots + k_{d-1}x_d )}
 = \begin{cases}
b^d & \text{if } x\in [0,b^{-d})\\
0& \text{else}.
\end{cases}
  \end{align*}
Since Walsh functions satisfy $\omega_k(x)\overline{\omega_k(y)}= \omega_k(x \ominus_b y)$, we obtain
\begin{equation*}
\sum_{k=0}^{b^{d}-1}\omega_k(x)\overline{\omega_k(y)} = \begin{cases} b^d & \mbox{if } x\ominus_b y \in [0,b^{-d}) \\ 0 & \mbox{otherwise},   \end{cases}
\end{equation*}
where the condition is equivalent to $ x, y \in [a b^{-d}, (a+1) b^{-d}) \mbox{ for some } 0 \le a < b^d$. For higher dimensional Walsh functions, we hence obtain
\begin{align}\label{eq:counting}
 \sum_{\bk \in K_\bd} \bomega_{\bk}(x) \overline{\bomega_{\bk}(y)} 
 &=
 \begin{cases} b^{|\bd|} & \mbox{if } x,y  \in I_{\ba,\bd}\\ 0 & \mbox{otherwise}   \end{cases}
\end{align}
for some $\ba \in \N_0^s$ with $a_j\in \{0,\ldots,b^{d_j}-1\}$.
From this, we obtain
\begin{align*}
 \phi_{K_\bd}(\bz) = & \sum_{\bk\in K_\bd} \frac{1}{N} \sum_{n=1}^N c_n \bomega_\bk(\bx_n) \overline{ \bomega_\bk(\bz) } 
 =  \frac{ b^{|\bd|} }{N} \sum_{n=1}^N \sum_{\ba\in K_\bd} c_n \b{1}_{\bx_n,\bz\in I_{\ba,\bd}}  =  \frac{ b^{|\bd|} }{N} \sum_{n=1 \atop \bx_n \in I_{\bd}(\bz) }^N c_n. % \frac{b^{|\bd|}}{N} |\XX_{\bd}(\bz)|.
\end{align*}
The second statement follows from $\sum_{n=1 \atop \bx_n \in I_{\bd}(\bz)}^N 1 = |\XX_{\bd}(\bz)|$. This concludes the proof.
\end{proof}

The function $\phi_K$ is in a sense an approximation of the indicator function. If $K = K_{\bd}$, then as $d_j \to \infty$ for all $j$, the function $\phi_{K_{\bd}} b^{-|\bd|}$ converges to the number of points $\bx_n$ for which $\bx_n = \bz$ and otherwise the function is $0$. The function $\phi_K$ relaxes the equality condition to elementary intervals. We choose $K = K_\nu =\bigcup_{\bd \in \N_0^s\atop |\bd|=\nu} K_\bd$ in~\eqref{eq:deriv1} and use the inclusion-exclusion formula \eqref{inexform} to obtain
\begin{align}\label{eq:incexc}
 \phi_{K_\nu}(\bz) = & \sum_{\bk \in K_\nu} \frac{1}{N} \sum_{n=1}^N c_n \bomega_{\bk}(\bx_n) \overline{\bomega_{\bk}(\bz)} \nonumber \\  = & \sum_{q=0}^{s-1} (-1)^q {s-1 \choose q} \sum_{\bd \in \N_0^s \atop |\bd| - \nu-q} \sum_{\bk \in K_{\bd}} \frac{1}{N} \sum_{n=1}^N c_n \bomega_{\bk}(\bx_n) \overline{\bomega_{\bk}(\bz)} \nonumber \\ = & \sum_{q=0}^{s-1} (-1)^q \begin{pmatrix}s-1 \\ q\end{pmatrix}\sum_{\bd \in \N_0^s\atop |\bd|=\nu-q}\phi_{K_\bd}{(\bz)}.
\end{align} 
Using the approximation \eqref{eq:deriv1} together with Lemma~\ref{lem:indicator} this results in
\begin{align*}
 \frac{1}{N} \sum_{n=1}^N c_n g(\bx_n)  \approx & \frac{1}{b^{m}} \sum_{\ell=0}^{b^{m}-1} g(\bz_\ell) \phi_{K_\nu}(\bz_\ell) \\ = & \frac{1}{b^{m}} \sum_{\ell=0}^{b^{m} - 1} g(\bz_\ell) \sum_{q=0}^{s-1} (-1)^q \begin{pmatrix}s-1\\ q\end{pmatrix}\sum_{\bd\in \N_0^s\atop |\bd|=\nu-q} \phi_{K_\bd}(\bz_\ell)\\
 = & \sum_{\ell= 0}^{b^{m} - 1}g(\bz_\ell) \frac{b^{\nu-m}}{N} \sum_{q=0}^{s-1} (-1)^q \begin{pmatrix}s-1\\ q\end{pmatrix} \frac{1}{b^q}  \sum_{\bd \in \N_0^s\atop |\bd|=\nu-q}  \sum_{n=1 \atop \bx_n \in I_{\bd}(\bz_\ell)}^N c_n \\
 = & \begin{cases} \frac{1}{b^{m}} \sum_{\ell= 0}^{b^{m}-1} f_\theta^2(\bz_\ell) W_{\XX, P, \nu, \ell} & \mbox{if } c_n = 1, g=f_\theta^2, \\ \frac{1}{b^{m}} \sum_{\ell= 0}^{b^{m} - 1} f_\theta(\bz_\ell) W_{\XX, \YY, P, \nu, \ell} & \mbox{if } c_n = y_n, g=f_\theta; \end{cases}
\end{align*}
These formulae for the weights are the same as the formulae for digital nets in \eqref{dignet_form} and \eqref{dignet_form_y} which we obtained using geometrical arguments.

\subsection{Error Analysis}

To analyse the error, we need to assume that the predictor $f_\theta$ has sufficient smoothness. More precisely, for $g=f_\theta$ and $g=f_\theta^2$, we require the following norm to be bounded
\begin{align*}
\norm{g}{p, \alpha}^p :=&\sum_{ u \subseteq \{1,\ldots,s\}} \sum_{v \subseteq u} \sum_{\tau\in \{1,\ldots,\alpha-1\}^{|u \setminus v|}}\\
&\int_{[0,1]^{|v|}} \Big| \int_{[0,1]^{s-|v|}} \big(\prod_{j\in v}\partial_{z_j}^\alpha \prod_{j\in u \setminus v}\partial_{z_j}^{\tau_j} \big) g(\bz)\,d\bz_{\{1,\ldots,s\}\setminus v}\Big|^p \,d\bz_{v},
\end{align*}
for some integer $\alpha \ge 2$ and any $1 \le p \le \infty$, with the obvious modifications for $p = \infty$. The notation $\big(\prod_{j\in v}\partial_{z_j}^\alpha \prod_{j\in u \setminus v}\partial_{z_j}^{\tau_j} \big) g(\bz)$ denotes the partial mixed derivative of order $\alpha_j$ or $\tau_j$ in coordinate $j$. This definition is a standard assumption for estimates regarding high-order QMC point
sets and is routinely satisfied for many problems appearing in the field of uncertainty quantification, see, e.g.~\cite{qmcsoa}.

The derivation of the approximation method in Section~\ref{ssec_W} shows that it suffices to control
\begin{align}\label{eq:err1}
{\rm err}_1:=\norm{g -g_{K_\nu}}{L^\infty([0,1]^s)}
\end{align}
to bound the first approximation error in~\eqref{eq:deriv1} (see also \eqref{eq:deriv}) as well as
\begin{align}\label{eq:err2}
{\rm err}_2:= \left|\int_{[0,1]^s} g(\bz)\phi_{K_\nu}(\bz) \,d \bz - \frac{1}{b^{m}} \sum_{\ell=0}^{b^{m} - 1} g(\bz_\ell)\phi_{K_\nu}(\bz_\ell)\right|
\end{align}
to bound the second approximation error in~\eqref{eq:deriv1}. We prove a bound on ${\rm err}_1$ in the following lemma.

\begin{lemma}
Assume that $\|g\|_{p,\alpha} < \infty$ for some integer $\alpha \ge 2$ and some $1 \le p \le \infty$. Then
\begin{equation*}
{\rm err}_1 \lesssim \|g\|_{p, \alpha} \nu^{2s-1} b^{-\nu}.
\end{equation*}
\end{lemma}

\begin{proof}
To control~\eqref{eq:err1}, we note that, under the given assumptions,~\cite[Theorem~14.23]{DP10} shows that there holds  
\begin{align*}
&\norm{g-g_{K_\nu}}{L^\infty([0,1]^s)}\lesssim  \|g\|_{p, \alpha}  \sum_{\bk\in\N_0^s\setminus K_\nu} b^{-\mu_{2}(\bk)} \\ 
&\qquad = 
\|g\|_{p, \alpha} \sum_{n_1=\nu+1}^\infty \sum_{a_{1,1}+\ldots +{a_{1,s}}=n_1\atop a_{1,1},\ldots,a_{1,s}\geq 0} 
\sum_{n_2=0}^{n_1-1} \sum_{a_{2,1}+\ldots +a_{2,s}=n_2\atop 0\leq a_{2,i}<a_{1,i},\,i=1,\ldots,s} \sum_{\bk\in\N_0^s \atop k_j = (a_{j,1},a_{j,2},\ldots)_b} b^{-n_1-n_2},
\end{align*}
where $(a_{j,1}, a_{j,2}, \ldots)_b$ signifies the position of the non-zero digits of $k_j$, i.e., $k_j = \kappa_{j,1} b^{a_{j,1}-1} + \kappa_{j,2} b^{a_{j,2}-1} + k'_j$, for some $\kappa_{j,1}, \kappa_{j,2} \in \{1, 2, \ldots, b-1\}$ and $0 \le k'_j < b^{a_{j,2}-1}$. The constant is independent of $g$ and $\nu$. 

If $k_j$ has at least two non-zero digits, there are  $b^{a_{j,2}-1} (b-1)^2$ numbers $k_j \in\N$ with $k_j= (a_{j,1},a_{j,2},\ldots)_b$. If $k_j$ has exactly one non-zero digit, there are $b-1$ choices and if $k_j = 0$ there is only one choice. Hence the above simplifies to
	\begin{align*}
	&\norm{g-g_{K_\nu}}{L^\infty([0,1]^s)}\lesssim \|g\|_{p, \alpha}
	\sum_{n_1=\nu+1}^\infty \sum_{a_{1,1}+\ldots +a_{1,s}=n_1\atop a_{1,1},\ldots,a_{1,s}\geq 1} 
	\sum_{n_2=0}^{n_1-1} \sum_{a_{2,1}+\ldots +a_{2,s}=n_2\atop 0\leq a_{2,i}<a_{1,i},\,i=1,\ldots,s}  b^{- n_1}.
	\end{align*}
	The cardinality of the set $\{(a_{j,1}, \ldots, a_{j,s}): a_{j,1}+\ldots+a_{j,s}=n_j, a_{j,i} \ge 0\}$ is bounded by $\binom{n_j+1}{s-1}\lesssim (n_j+1)^{s-1}$. Hence, there holds
	\begin{align*}
	\norm{g-g_{K_\nu}}{L^\infty([0,1]^s)}&\lesssim \|g\|_{p, \alpha}
 	\sum_{n_1=\nu+1}^\infty (n_1+1)^{s-1} b^{-n_1} 
 	\sum_{n_2=0}^{n_1-1}  (n_2+1)^{s-1} \\ & \lesssim  \|g\|_{p, \alpha}	\sum_{n=\nu+1}^\infty (n+1)^{2s-1} b^{-n} 
 	 \lesssim \|g\|_{p, \alpha} \nu^{2s-1}b^{-\nu},
	\end{align*}
	where the constant is independent of $g$ and $\nu$.
\end{proof}
In the following we deal with the integration error ${\rm err}_2$ defined in \eqref{eq:err2}. First we use digital $(t,m,s)$-nets which achieve almost order one convergence, and in the subsequent section we deal with order $\alpha$ digital $(t,m,s)$-nets which achieve almost order $\alpha$ convergence of the integration error with the usual drawbacks such as stronger dependence of the constants on data dimension and higher $t$-value.

\subsubsection{Order one convergence}

We consider the reproducing kernel
\begin{align*}
K: [0,1]^s \times [0,1]^s \mapsto & \mathbb{R} \\
    K(\by, \bz) = & \prod_{j=1}^s (1 + \min\{1-y_j, 1-z_j\}),
\end{align*}
which defines a reproducing kernel Hilbert space of functions $f, g: [0,1]^s \to \mathbb{R}$ with inner product
\begin{equation*}
\langle f, g \rangle = \sum_{u \subseteq \{1, \ldots, s\}} \int_{[0,1]^{|u|}} \partial_{\bz_u} f(\bz_u, \boldsymbol{1}_{-u}) \partial_{\bz_u} g(\bz_u, \boldsymbol{1}_{-u}) \,\mathrm{d} \bz_u.
\end{equation*}
We consider the function space
\begin{equation*}
H = \{g: [0,1]^s {\to} \mathbb{R}:  \partial_{\bz_u} g(\cdot, \boldsymbol{1}_{-u}) \in L_1([0,1]^{|u|}) \mbox{ for all } u \subseteq \{1, \ldots, s\} \}.
\end{equation*}
For functions $g \in H$ we define the norm
\begin{equation*}
\| g \| = \sum_{u \subseteq \{1, \ldots, s\}} \int_{[0,1]^{|u|}} \left| \partial_{\bz_u} g(\bz_u, \boldsymbol{1}_{-u}) \right| \,\mathrm{d} \bz_u.
\end{equation*}
With these definitions we have for any $g \in H$ that
\begin{equation*}
g(\bz) = \langle g, K(\cdot, \bz) \rangle, \quad \bz \in [0,1]^s.
\end{equation*}

\begin{theorem}\label{thm:err1c}
Assume that $g \in H$ and $\norm{g}{1, 2} < \infty$. Let $\{\bz_0, \bz_1, \ldots, \bz_{b^{m}-1}\}$ be a digital $(t,m,s)$-net in base $b$. Then, there holds for $m-t \geq \nu \in\N$
\begin{align*}
& \left| \frac{1}{N}\sum_{n = 1}^N c_n g(\bx_n) - \frac{1}{b^{m}} \sum_{\ell=0}^{b^{m} - 1} g(\bz_\ell) \sum_{q=0}^{s-1} (-1)^q \begin{pmatrix}s-1\\ q\end{pmatrix}\sum_{\bd\in \N_0^s\atop |\bd|=\nu-q} \phi_{K_\bd}(\bz_\ell) \right| \\ \lesssim & \|g\|  \nu^{s-1} m^{s-1}  b^{-  (m-\nu)} \frac{1}{N} \sum_{n=1}^N |c_n| + \|g\|_{1,2} \nu^{2s-1} b^{-\nu},
\end{align*}
for some constant independent of $N, m, \nu, g, \{c_n\}_n$.
\end{theorem}

\begin{proof}
It remains to prove a bound on ${\rm err}_2$. To do so we use the Koksma-Hlawka inequality
\begin{equation*}
\left| \int_{[0,1]^s} g(\bz) \phi_{K_\nu}(\bz) \rd \bz - \frac{1}{N} \sum_{\ell=0}^{L-1} g(\bz_\ell) \phi_{K_\nu}(\bz_\ell) \right| \le D^\ast(\{\bz_0, \bz_1, \ldots, \bz_{L-1}\}) V(g \phi_{K_\nu} ),
\end{equation*}
where $V(g \phi_{K_\nu})$ is the variation of $g \phi_{K_\nu}$ in terms of Hardy and Krause (defined below) and $D^\ast(\{\bz_0, \ldots, \bz_{L-1}\})$ is the star-discrepancy of the digital $(t,m,s)$-net. It is known that $D^\ast(\{\bz_0, \ldots, \bz_{L-1}\})$ is of order $m^{s-1} b^{-m}$ (see \cite[Theorem~5.1, 5.2]{DP10}). Hence it remains to prove a bound on the Hardy and Krause variation of $g \phi_{K_\nu}$, which we define in the following.

Let $J = \prod_{j=1}^s [a_j, b_j) \subseteq [0,1)^s$ be a subinterval and let 
\begin{equation*}
\Delta(g;J) = \sum_{u \subseteq \{1, 2, \ldots, s\}} (-1)^{|u|} g(\ba_u, \boldsymbol{b}_{-u}),
\end{equation*}
where $(\ba_u, \boldsymbol{b}_{-u})$ is the vector whose $j$th component is $a_j$ if $j \in u$ and $b_j$ if $j \notin u$. The variation of a function in the sense of Vitali is defined by
\begin{equation*}
V^{(s)}(g) = \sup_{\mathcal{P}} \sum_{J \in \mathcal{P}} |\Delta(g;J)|,
\end{equation*}
where the supremum is extended over all partitions $\mathcal{P}$ of $[0,1)^s$ into subintervals. For example, for a Walsh function $\omega_k$ we have $V^{(1)}(\omega_k) \le b^{\mu(k)}$, and for a Walsh function $\bomega_{\bk}$ we have $V^{(s)}(\bomega_{\bk}) \le b^{\mu(\bk)}$, since $\bomega_{\bk}$ is a piecewise constant function which is constant on elementary intervals of the form $I_{\ba, \mu(\bk)}$. More generally we have
\begin{equation*}
V^{(s)}(\phi_{K_\nu}) \lesssim {s+\nu-1 \choose s-1} \frac{b^\nu}{N} \sum_{n=1}^N |c_n|,
\end{equation*}
since $\phi_{K_\nu}$ can be written as a sum of $\sum_{|\bd|= \nu} \phi_{K_{\bd}}$, with $\phi_{K_{\bd}}$ satisfying $V^{(s)}(\phi_{K_{\bd}}) \le b^{\nu} \frac{1}{N} \sum_{n=1}^N |c_n|$, where $\frac{1}{N} \sum_{n=1}^N |c_n|$ provides a bound on the maximum change of each discontinuity of the piecewise constant functions.

For $1 \le k \le s$ and $1 \le i_1 < i_2 < \cdots < i_k \le s$, let $V^{(k)}(f; i_1, \ldots, i_k)$ be the variation in the sense of Vitali of the restriction of $f$ to the $k$-dimensional face $\{(x_1, \ldots, x_s) \in [0,1]^s: x_j = 1 \mbox{ for } j \neq i_1, i_2, \ldots, i_k\}$. Then the variation of $f$ in the sense of Hardy and Krause is defined by
\begin{equation*}
V(g) = \sum_{k=1}^s \sum_{1 \le i_1 < i_2 < \cdots < i_s \le s} V^{(g)}(f; i_1, i_2, \ldots, i_s).
\end{equation*}
Again we have
\begin{equation*}
V(\phi_{K_\nu}) \lesssim {s+\nu-1 \choose s-1} \frac{b^\nu}{N} \sum_{n=1}^N |c_n|,
\end{equation*}
where the constant only depends on the dimension $s$.

Now consider the product $g \phi_{K_\nu}$. Let $J = \prod_{j=1}^s [a_j, b_j)$, $0 \le a_j < b_j \le 1$ be an interval. Then using the representation $g(\bz) = \langle g, K(\cdot, \bz) \rangle$ we obtain
\begin{align*}
\Delta(g \phi_{K_\nu}; J)  = & \sum_{u \subseteq \{1, \ldots, s\}} (-1)^{|u|} g(\ba_u, \boldsymbol{b}_{-u}) \phi_{K_\nu}(\ba_u, \boldsymbol{b}_{-u})  \\ = & g(\boldsymbol{1}) \sum_{u \subseteq \{1, \ldots, s\}} (-1)^{|u|} \phi_{K_\nu}(\ba_u, \boldsymbol{b}_{-u})  \\ & +  \sum_{\emptyset \neq v \subseteq\{1,\ldots, s\}} \sum_{u \subseteq \{1, \ldots, s\}} (-1)^{|u|} \phi_{K_\nu}(\ba_u, \boldsymbol{b}_{-u}) \\ & \int_{[0,1]^{|v|}} \partial_{\bz_v} g(\bz_v, \boldsymbol{1}_{-v}) \partial_{\bz_v} K((\bz_v, \boldsymbol{1}_{-v}), (\ba_u, \boldsymbol{b}_{-u}){)} \,\mathrm{d} \bz_v.
\end{align*}
We have
\begin{equation*}
\partial_{\bz_v} K((\bz_v, \boldsymbol{1}_{-v}), (\ba_u, \boldsymbol{b}_{-u}) {)}= \begin{cases} (-1)^{|v|} & \mbox{if } \bz_v > (\ba_u, \boldsymbol{b}_{-u})_v, \\ 0 & \mbox{if } \bz_v < (\ba_u, \boldsymbol{b}_{-u})_v, \end{cases}
\end{equation*}
where $\bz_v > (\ba_u, \boldsymbol{b}_{-u})_v$ means that for all $j \in v$ we have $z_j > a_j$ if $j \in u$ and $z_j > b_j$ if $j \notin u$. Substituting this into the last equation we obtain
\begin{align*}
& |\Delta(g \phi_{K_\nu}; J) |  = \left| \sum_{u \subseteq \{1, \ldots, s\}} (-1)^{|u|} g(\ba_u, \boldsymbol{b}_{-u}) \phi_{K_\nu}(\ba_u, \boldsymbol{b}_{-u}) \right|  \\ \le & | g(\boldsymbol{1}) | \left| \sum_{u \subseteq \{1, \ldots, s\}} (-1)^{|u|} \phi_{K_\nu}(\ba_u, \boldsymbol{b}_{-u}) \right| \\ & + \left|  \sum_{\emptyset \neq v \subseteq\{1,\ldots, s\}} (-1)^{|v|} \int_{[0,1]^{|v|}} \partial_{\bz_v} g(\bz_v, \boldsymbol{1}_{-v}) \sum_{u \subseteq \{1, \ldots, s\}} (-1)^{|u|} \phi_{K_\nu}(\ba_u, \boldsymbol{b}_{-u})  \boldsymbol{1}_{\bz_v > (\ba_u, \boldsymbol{b}_{-u})_v} \,\mathrm{d} \bz_v \right| \\ \le & | g(\boldsymbol{1}) | \left| \sum_{u \subseteq \{1, \ldots, s\}} (-1)^{|u|} \phi_{K_\nu}(\ba_u, \boldsymbol{b}_{-u}) \right| \\ & +  \sum_{\emptyset \neq v \subseteq\{1,\ldots, s\}} \int_{[0,1]^{|v|}} | \partial_{\bz_v} g(\bz_v, \boldsymbol{1}_{-v}) | \left| \sum_{u \subseteq \{1, \ldots, s\}} (-1)^{|u|} \phi_{K_\nu}(\ba_u, \boldsymbol{b}_{-u}) \boldsymbol{1}_{\bz_v > (\ba_u, \boldsymbol{b}_{-u})_v} \right|  \,\mathrm{d} \bz_v.
\end{align*}
Let $\mathcal{P}$ be a partition of $[0,1)^s$ into intervals of the form $J$. Then
\begin{align*}
& \sum_{J \in \mathcal{P}} |\Delta(g \phi_{K_\nu}; J) | \le  |g(\boldsymbol{1})| V(\phi_{K_\nu}) \\ &  +   \sum_{\emptyset \neq v \subseteq\{1,\ldots, s\}} \int_{[0,1]^{|v|}} | \partial_{\bz_v} g(\bz_v, \boldsymbol{1}_{-v}) | \sum_{J \in \mathcal{P}} \left| \sum_{u \subseteq \{1, \ldots, s\}} (-1)^{|u|} \phi_{K_\nu}(\ba_u, \boldsymbol{b}_{-u}) \boldsymbol{1}_{\bz_v > (\ba_u, \boldsymbol{b}_{-u})_v} \right|  \,\mathrm{d} \bz_v. %\\ \le & V(\phi_{K_\nu}) \left(|g(\boldsymbol{1}| + \sum_{\emptyset \neq v \subseteq \{1, \ldots, s\}} \int_{[0,1]^{|v|}} |\partial_{\bz_v} g(\bz_v, \boldsymbol{1}_{-v})| \, \mathrm{d} \bz_v \right) = V(\phi_{K_\nu}) \|g\|.
\end{align*}
Similar to $V(\phi_{K_\nu})$, we can also estimate
\begin{equation*}
\sum_{J \in \mathcal{P}} \left| \sum_{u \subseteq \{1, \ldots, s\}} (-1)^{|u|} \phi_{K_\nu}(\ba_u, \boldsymbol{b}_{-u}) \boldsymbol{1}_{\bz_v > (\ba_u, \boldsymbol{b}_{-u})_v} \right| \le {\nu+s-1 \choose s-1} \frac{b^\nu}{N} \sum_{n=1}^N |c_n|.
\end{equation*}
The result now follows by combining this bound with the bound on $V(\phi_{K_\nu})$, the Koksma-Hlawka inequality and the bound on the discrepancy for digital nets.
\end{proof}

Using \eqref{dignet_form} and \eqref{dignet_form_y} in Theorem~\ref{thm:err1c} we obtain the following approximation.
\begin{corollary}\label{corc}
Let $P$ be a digital $(t, m, s)$-net in base $b$. Choose the integer $\nu$ such that $\nu \le m-t$. Assume that $f_\theta, f^2_\theta \in H$ and $\|f_\theta\|_{1, 2}, \|f_\theta^2\|_{1, 2} < \infty$ for all parameters $\theta$. Then
\begin{align*}
\left| {\rm err}(f_\theta) - {\rm app}_{b^m}(f_\theta) \right| & \lesssim  \|f_\theta^2\| \nu^{s-1} m^{s-1} b^{- (m-\nu)} + \|f_\theta^2\|_{1,2} \nu^{2s-1} b^{-\nu}  \\ & + \|f_\theta\| \nu^{s-1} m^{s} b^{- (m-\nu)} \frac{1}{N} \sum_{n=1}^N |y_n| + \|f_\theta\|_{1,2} \nu^{2s-1} b^{-\nu},
\end{align*}
for some constant independent of $N, m, \nu, \{y_n\}_n, f_\theta$.
\end{corollary}
 In order to balance the error  one should choose $m-\nu \approx \nu$, which implies that $\nu \approx m / 2$. Hence, overall we get an error of order $m^{2(s-1) } b^{-m / 2}$, which corresponds to $\mathcal{O}( \log(L)^{2(s-1)}/\sqrt{L})$. This might not seem like any improvement over Monte Carlo type methods, however, we note that comparable methods do not provide non-probabilistic error estimates.

\subsubsection{Higher order convergence}

In this section we prove bounds on the error using higher order digital nets, which yields higher rates $\mathcal{O}(L^{-1+\eps})$ of convergence provided that the predictor $f_\theta$ satisfies some smoothness assumptions.

\begin{theorem}\label{thm:err1}
Let $\norm{g}{2,\alpha}<\infty$ for some integer $\alpha \ge 2$. Let $\{\bz_0, \ldots, \bz_{b^{m}-1}\}$ be an order $\alpha$ digital $(t_\alpha,m,s)$-net in base $b$. Then, there holds for $m-t_\alpha \geq \nu \in\N$
\begin{align*}
& \left| \frac{1}{N}\sum_{n = 1}^N c_n g(\bx_n) - \frac{1}{b^{m}} \sum_{\ell=0}^{b^{m} - 1} g(\bz_\ell) \sum_{q=0}^{s-1} (-1)^q \begin{pmatrix}s-1\\ q\end{pmatrix}\sum_{\bd\in \N_0^s\atop |\bd|=\nu-q} \phi_{K_\bd}(\bz_\ell) \right| \\ \lesssim & \|g\|_{2, \alpha} \left( m^{\alpha s}  b^{- \alpha (m-\nu)} \frac{1}{N} \sum_{n=1}^N |c_n| + \nu^{2s-1} b^{-\nu} \right),
\end{align*}
for some constant independent of $N, m, \nu, g, \{c_n\}_n$.
\end{theorem}

\begin{proof}
It remains to prove a bound on ${\rm err}_2$. Assume that $g = \sum_{\bk\in \N_0^s} \alpha_\bk \bomega_\bk$ for some Walsh-coefficients $\alpha_\bk\in \R$. By the definition of the Walsh functions, there holds
\begin{align*}
g \phi_{K_\nu} = \sum_{\bk\in \N_0^s\atop \bk'\in K_\nu} \alpha_\bk \mu_{\bk'} \bomega_{\bk\ominus \bk'} = \sum_{\bk \in \mathbb{N}_0^s} \left(\sum_{\bk' \in K_\nu} \alpha_{\bk \oplus \bk'} \mu_{\bk'} \right) \bomega_{\bk},
\end{align*}
which means that the $\bk$th Walsh coefficient $\beta_{\bk}$ of $g \phi_{K_\nu}$ reads $\beta_{\bk} = \sum_{\bk' \in K_\nu} \alpha_{\bk \oplus \bk'} \mu_{\bk'}$.

The regularity assumption on $g$ and~\cite[Theorem~14.23]{DP10} imply that $$|\alpha_\bk|\lesssim b^{-\mu_\alpha(\bk)} \norm{g}{2,\alpha},$$ where the constant is independent of $\bk$ and $g$. For the Walsh coefficients of $\phi_{K_\nu}$, we only know $$|\mu_\bk| \lesssim \frac{1}{N} \sum_{n=1}^N |c_n|$$ by the definition of $\mu_\bk$ in~\eqref{eq:defmu}. Thus, for the Walsh-expansion of the product $g \phi_{K_\nu}$, we know that the coefficients behave like 
\begin{align*}
\left| \beta_{\bk} \right| = & \left| \sum_{\bk' \in K_\nu} \alpha_{\bk \oplus \bk'} \mu_{\bk'} \right|  \le   \sum_{\bk' \in K_\nu} |\alpha_{\bk \oplus \bk'} | |\mu_{\bk'}|  \lesssim  \|g\|_{2,  \alpha} \frac{1}{N} \sum_{n=1}^N |c_n| \sum_{\bk' \in K_\nu} b^{-\mu_\alpha(\bk \oplus \bk')} \\ \lesssim & \|g\|_{2, \alpha} \frac{1}{N} \sum_{n=1}^N |c_n| \,\, b^{-\mu_{\alpha}(\bk) + \alpha \nu},
\end{align*}
where the constant is independent of $\bk$ and $g$ and {the last inequality follows from the fact that $\mu_\alpha(\bk \oplus \bk') \geq \mu_\alpha(\bk) +\mu_\alpha(\bk')- \nu\alpha  $
and that $\sum_{\bk' \in K_\nu} b^{-\mu_\alpha(\bk')}<\infty$ uniformly in $\nu$.} Thus, arguing as in the proof of \cite[Theorem~15.21]{DP10}, we obtain
\begin{align*}
 {\rm err}_2\lesssim (m)^{\alpha s}b^{-\alpha(m-\nu)}\norm{g}{2,\alpha} \frac{1}{N} \sum_{n=1}^N |c_n|,
\end{align*}
where the constant is independent of $m, \nu, \{c_n\}_n, g$. This concludes the proof.
\end{proof}

Using \eqref{dignet_form} and \eqref{dignet_form_y} in Theorem~\ref{thm:err1} we obtain the following approximation.
\begin{corollary}\label{cor}
Let $\alpha \ge 2$ be an integer and let $P$ be an order $\alpha$ digital $(t, m, s)$-net in base $b$. Choose the integer $\nu$ such that $\nu \le m-t$. Let $\|f_\theta\|_{2, \alpha}, \|f_\theta^2\|_{2, \alpha} < \infty$ for all parameters $\theta$. Then
\begin{align*}
\left| {\rm err}(f_\theta) - {\rm app}_{b^m}(f_\theta) \right| & \lesssim  \|f_\theta^2\|_{2, \alpha} \left( m^{\alpha s} b^{-\alpha (m-\nu)} + \nu^{2s-1} b^{-\nu} \right) \\ & + \|f_\theta\|_{2, \alpha} \left(m^{\alpha s} b^{-\alpha(m-\nu)} \frac{1}{N} \sum_{n=1}^N |y_n| + \nu^{2s-1} b^{-\nu} \right),
\end{align*}
for some constant independent of $N, m, \nu, \{y_n\}_n, f_\theta$.
\end{corollary}
 In order to balance the error  one should choose $\alpha (m-\nu) \approx \nu$, which implies that $\nu \approx \frac{\alpha}{\alpha+1} m$. Hence, overall we get an error of order $m^{{\alpha}s} b^{-\frac{\alpha}{\alpha+1} m}$.
\begin{remark}\label{rem:cost}
 Corollary~\ref{cor} together with Theorem~\ref{thm:cost2} show that for $(t,m,s)$-nets with moderate $t$-value, the choice $\nu\simeq m/(1+1/\alpha)$ leads to an error bound 
 \begin{align*}
  {\rm error}=\mathcal{O}(b^{-\alpha m/(1+\alpha)})
 \end{align*}
with a startup cost of
\begin{align*}
 {\rm cost}_{\rm startup} = \mathcal{O}(Nb^{m/(1+\alpha)})
\end{align*}
and an online cost (same data points but different values of $\theta$) of
\begin{align*}
 {\rm cost} = \mathcal{O}(b^{m}),
\end{align*}
where the hidden constants depend exponentially on the dimension but only polynomially on $m$.
\end{remark}

\subsection{Approximation of parameters}

In Corollary~\ref{corc} and Corollary~\ref{cor} we have shown that our data compression method yields an approximation of the squared error. In practice one may be interested in how this approximation changes the choice of parameters. This is a well studied problem in optimization called 'perturbation analysis', see for instance \cite{BS00}. In this section we apply \cite[Proposition~4.32]{BS00} to obtain such a result.

Assume that the optimization problem
\begin{equation*}
\min_{\theta \in \Omega} \mathrm{err}(f_\theta)
\end{equation*}
has a non-empty solution set $S_0$. We say that $\mathrm{err}(f_\theta)$ satisfies the second order growth condition at $S_0$ if there exists a neighborhood $N$ of $S_0$ and a constant $c > 0$ such that
\begin{equation*}
\mathrm{err}(f_\theta) \ge E_0 + c [\mathrm{dist}(x, S_0)]^2, \quad \forall x \in \Omega \cap N,
\end{equation*}
where $E_0 = \inf_{\theta \in \Omega} \mathrm{err}(f_\theta)$ and $\rm{dist}$ is the Euclidean distance.

In many optimization problems one needs to use an approximation algorithm to find an approximate solution. We call $\overline{\theta}$ an $\varepsilon$-solution of $\min_{\theta \in \Omega} \mathrm{err}(f_\theta)$ if 
\begin{equation*}
\mathrm{err}(f_{\overline{\theta}}) \le E_0 + \varepsilon.
\end{equation*}

Define the function
\begin{equation*}
\mathcal{A}(\theta) = \mathrm{err}(f_\theta) - \mathrm{app}_{b^m}(f_\theta).
\end{equation*}
In order to obtain a result on how much the optimal parameter changes by switching from $\mathrm{err}(f_\theta)$ to $\mathrm{app}_{b^m}(f_\theta)$, we need a bound on the gradient $\nabla_\theta \mathcal{A}(\theta)$. We can use Theorem~\ref{thm:err1} to obtain such a result.

Define
\begin{equation*}
A_{p, \alpha} = \max_j \sup_{\theta \in \Omega} \left\|f_\theta \frac{\partial f_\theta}{\partial \theta_j} \right\|_{p, \alpha}
\end{equation*}
and
\begin{equation*}
B_{p,\alpha} = \max_j \sup_{\theta \in \Omega} \left\| \frac{ \partial f_\theta}{\partial \theta_j} \right\|_{p, \alpha}.
\end{equation*}

We can now use Theorem~\ref{thm:err1c} for the order $1$ case and  Theorem~\ref{thm:err1} for the order $\alpha \ge 2$ case, with $c_n = 1$ and $g(\bz) = f_\theta(\bz) \frac{\partial f_\theta(\bz)}{\partial \theta_j}$, assuming that $\|g\|_{2, \alpha}$ is bounded independently of $\theta$ for all $\theta \in \overline{\Omega}$. In the second step we set $c_n = y_n$ and $g(\bz) = \frac{\partial f_\theta(\bz)}{\partial \theta_j}$. Similarly to Corollary~\ref{corc} we obtain
\begin{align*}
\max_j \sup_{\theta \in \Omega} \left| \frac{\partial \mathcal{A}}{\partial \theta_j} \right|  & \lesssim  A_{1,1} \nu^{s-1} m^{s-1} b^{- (m-\nu)} + A_{1,2} \nu^{2s-1} b^{-\nu}  \\ & + B_{1,1} \nu^{s-1} m^{s} b^{- (m-\nu)} \frac{1}{N} \sum_{n=1}^N |y_n| + B_{1,2} \nu^{2s-1} b^{-\nu},
\end{align*}
and similarly to Corollary~\ref{cor} we obtain
\begin{align*}
\max_j \sup_{\theta \in \Omega} \left| \frac{\partial \mathcal{A}}{\partial \theta_j} \right|  \lesssim &  A_{2,\alpha} \left( m^{\alpha s} b^{-\alpha (m-\nu)} + \nu^{2s-1} b^{-\nu} \right) \\ &  + B_{2,\alpha} \left(m^{\alpha s} b^{-\alpha(m-\nu)} \frac{1}{N} \sum_{n=1}^N |y_n| + \nu^{2s-1} b^{-\nu} \right).
\end{align*}
This implies that $\mathcal{A}$ satisfies a Lipschitz condition with modulus $\kappa$ which satisfies
\begin{align*}
\kappa & \lesssim  A_{1,1} \nu^{s-1} m^{s-1} b^{- (m-\nu)} + A_{1,2} \nu^{2s-1} b^{-\nu}   \\
&\qquad+ B_{1,1} \nu^{s-1} m^{s} b^{- (m-\nu)} \frac{1}{N} \sum_{n=1}^N |y_n| + B_{1,2} \nu^{2s-1} b^{-\nu},
\end{align*}
if we use digital $(t,m,s)$-nets, and satisfies
\begin{equation*}
\kappa \lesssim  A_{2,\alpha} \left( m^{\alpha s} b^{-\alpha (m-\nu)} + \nu^{2s-1} b^{-\nu} \right)  + B_{2,\alpha} \left(m^{\alpha s} b^{-\alpha(m-\nu)} \frac{1}{N} \sum_{n=1}^N |y_n| + \nu^{2s-1} b^{-\nu} \right)
\end{equation*}
if we use order $\alpha$ digital $(t,m,s)$-nets.

The following result is \cite[Proposition~4.32]{BS00} applied to our situation.
\begin{theorem}[cf.~{\cite[Proof of Proposition~4.32]{BS00}}] 
Assume that $\mathrm{err}(f_\theta)$ satisfies the second order growth condition with constant $c > 0$ and that the function $\mathrm{err}(f_\theta)- \mathrm{app}_{b^m}(f_\theta)$ is Lipschitz continuous with modulus $\kappa$ on $\Omega \cap N$. Let $\overline{\theta}$ be an $\varepsilon$-solution of $\min_{\theta \in \Omega} \mathrm{app}_{b^m}(f_\theta)$. Then
\begin{equation*}
\mathrm{dist}(\overline{\theta}, S_0) \le c^{-1} \kappa + c^{-1/2} \varepsilon^{1/2}.
\end{equation*} 
\end{theorem}

\section{Numerical experiments}
\subsection{Linear regression}
We simulate a linear regression by testing the approximation quality of the method on the function
\begin{align*}
 f_\theta(x) = (1,x)\theta
\end{align*}
for a weight vector $\theta\in \R^{s+1}$ which is randomly generated.
 The data $\bx_i \in\R^s$, $i=1,\ldots,N$ was generated randomly by sampling a standard normal distribution and scaling the absolute value to the unit cube.  Analogously, we generated the labels $y_i\in [0,1]$, $i=1,\ldots,N$ randomly from a uniform distribution for a problem size  $N=10^6$ in $s=6$ dimensions. We  approximate the error
 \begin{align*}
{\rm err}(f_\theta):=\sum_{i=1}^N(f_\theta(\bx_i)-y_i)^2\approx {\rm app}_L(f_\theta).
\end{align*}
Figure~\ref{fig:regapprox} shows the convergence over 100 samples of weight vectors for different values of $\nu$ and $m$. {We note for a prescribed  accuracy of $\approx 2\cdot 10^{-3}$, the achieved compression rate ${\rm cost}/N$ is on average $\approx 1/10^{3}$.} We use higher-order Sobol points generated by interlacing~\cite{D08}.

\begin{figure}
\psfrag{cost}[cc][cc]{\tiny ${\rm cost}=2^m$, $m = \lceil(1+1/\alpha) \nu+4\rceil$}
\psfrag{1avg}[lc][cc]{\tiny $\alpha=1$, average}
\psfrag{1max}[lc][cc]{\tiny $\alpha=1$, maximum}
\psfrag{2avg}[lc][cc]{\tiny $\alpha=2$, average}
\psfrag{2max}[lc][cc]{\tiny $\alpha=2$, maximum}
\includegraphics[width=0.7\textwidth]{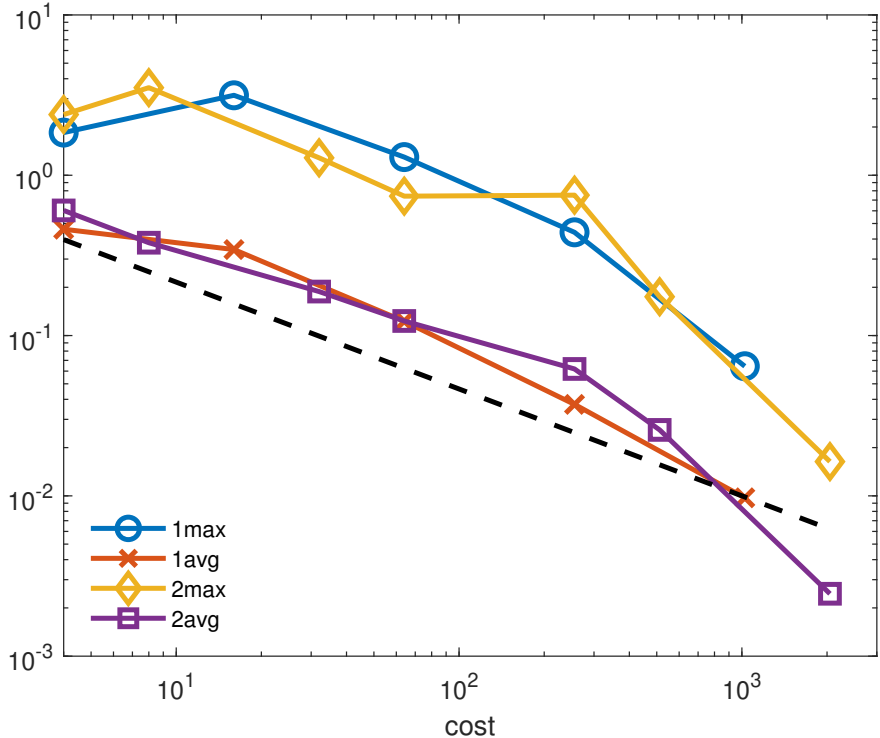}%
\caption{The average and maximal approximation error $|{\rm err}(f_\theta)-{\rm app}_{b^m}(f_\theta)|$ over 100 random samples of $\theta\in\R^6$. We use first and second order digital nets and choose $\nu$ accordingly to obtain the predicted convergence order of $\mathcal{O}({\rm cost}^{\alpha/(1+\alpha)})$ (indicated by the dashed line for $\alpha=2$).}
\label{fig:regapprox} 
\end{figure}

\subsection{Deep neural networks} 
We test the approximation quality on randomly generated deep neural networks by using the MNIST dataset of handwritten digits (see~\url{http://yann.lecun.com/exdb/mnist/}). We use  neural networks of the following layer/node structure: A shallow net defined by
\begin{align*}
\bx \mapsto\fbox{200}\longrightarrow\fbox{1}\mapsto f_\theta(\bx),
\end{align*}
as well as a deep net defined by
\begin{align*}
\bx \mapsto\fbox{200}\longrightarrow\fbox{100}\longrightarrow\fbox{50}\longrightarrow\fbox{20}\longrightarrow\fbox{1}\mapsto f_\theta(\bx),
\end{align*}
where $\theta$ is the vector of weights of the neural network, i.e., for the deep net
\begin{align*}
 f(\bx) = W_4\phi(W_3\phi(W_2\phi(W_1 \bx))),\quad \theta=(W_1,\ldots,W_4)
\end{align*}
for matrices $W_i\in \R^{n_{i+1}\times n_{i}}$, with $n=(200,100,50,20,1)$ and a given activation function $\phi\colon \R\to \R$.
The handwritten digits $x$ are 20x20 greyscale images which show digits from 0 to 9. The labels $\YY=(y_1,y_2,\ldots,y_N)$ contain the correct numbers from 0 to 9. We approximate the error
\begin{align*}
{\rm err}(f_\theta):=\sum_{i=1}^N(f_\theta(\bx_i)-y_i)^2\approx {\rm app}_{b^m}(f_\theta).
\end{align*}
The database contains more than $N=60000$ samples (a couple of them are shown in Figure~\ref{fig:mnist}). 
We subsample the images using a $2\times2$-stencil to reduce the input dimension to $s=100$. 
This allows us to still choose digital nets with reasonably bounded $t$-value for the problem sizes at hand. We use $t$-value optimized Sobol sequences from~\cite{sobolt} which can be downloaded 
from \texttt{https://web.maths.unsw.edu.au/\~{}fkuo/sobol}. For instance, Table~\ref{tval} (which is \cite[Table~3.8]{JK08}) shows the dimensions at which each $t$-value first occurs.
\begin{table}%\label{tval}
\begin{tabular}{c|c|c|c|c|c|c|c|c|c|c|c|c}
$m$ \textbackslash $t$ & 0 & 1 & 2 & 3 & 4 & 5 & 6 & 7 & 8 & 9 & 10 & 11 \\ \hline 
10 &        2 & 3 & 4 & 5 & 9 & 16 & 32 & 76 & 167 & 431 & $>$8300 \\
12 &        2 & 3 & 4 & 6 & 10 & 16 & 34 & 40 & 109 & 242 & 506 & 1049 \\
14 &        2 & 3 & 4 & 6 & 8 & 12 & 22 & 48 & 85 & 164 & 383 & 761 \\
16 &        2 & 3 & 4 & 6 & 8 & 14 & 15 & 35 & 80 & 159 & 280 & 525 \\
18 &        2 & 3 & 4 & 7 & 8 & 11 & 15 & 35 & 70 & 108 & 220 & 393 \\
\end{tabular}
\caption{This table lists the dimensions at which each $t$-value first occurs for Sobol$\cprime$ sequences in base $b = 2$. Taken from \cite[Table~3.8]{JK08}, which contains more values.}\label{tval}
\end{table} 
As activation function, we used the smooth sigmoid function $\phi(x):= \frac{2}{1+e^{-x}}-1$.

\begin{figure}
 \includegraphics[width=0.7\textwidth]{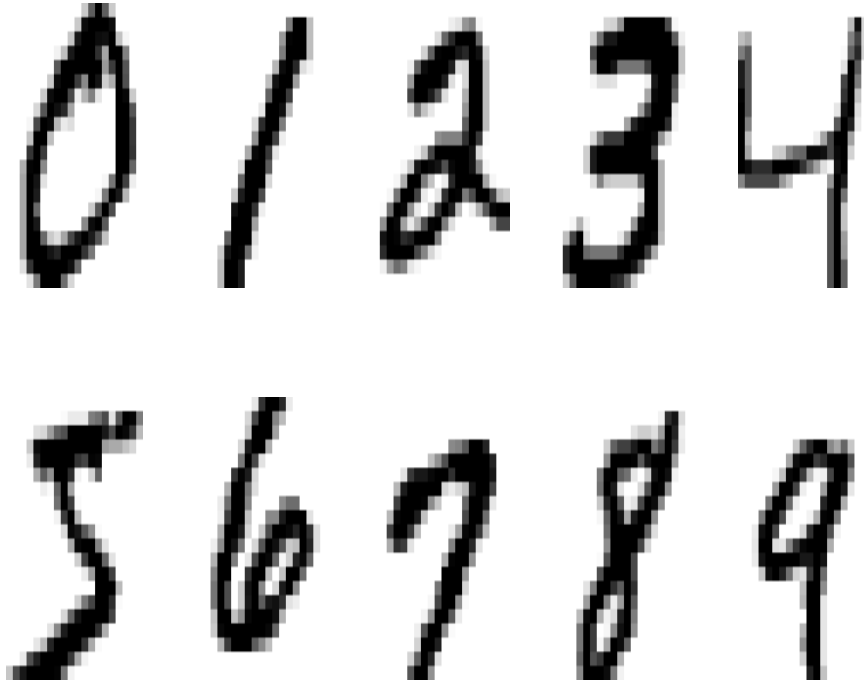}
 \caption{Sample images from the MNIST database of handwritten digits. Each pixel (20x20) contains a gray value between 0 and 1.}
 \label{fig:mnist}
\end{figure}
In Figure~\ref{fig:mnistapprox}, we plot the approximation error over hundred randomly generated weight-vectors $\theta$. 
We observe that compressing the data to
$2^{m}/N\approx 0.13$ of its original size yields an average compression loss of less
than ten percent. The large discrepancy between $\nu$ and $m$ is necessary to offset the fairly large $t$-value 
($\approx 10$) of the 100-dimensional Sobol sequence.
We also plot the distribution of the data points in some arbitrarily selected dimension to illustrate the non-trivial 
density which is implicitly approximated by the function $\phi_K$.
The method seems to be fairly robust with regard to the depth of the neural network as suggested by the similar results for the shallow and for the deep net.
\begin{figure}
\psfrag{density}[cc][cc]{\tiny percentage}
\psfrag{error}[cc][cc]{\tiny ${\rm err}(f_\theta)$}
\includegraphics[width=0.49\textwidth]{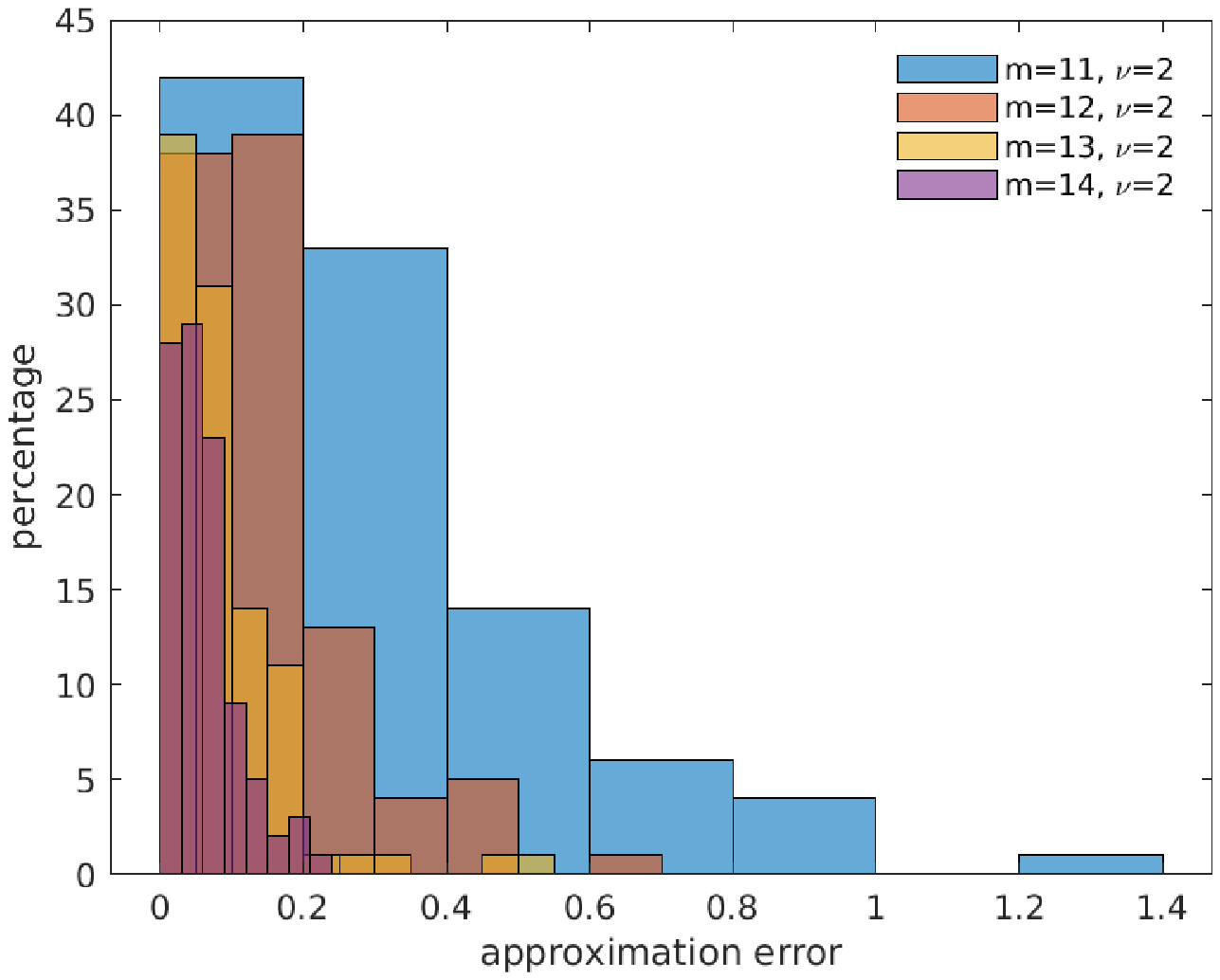}%
\includegraphics[width=0.49\textwidth]{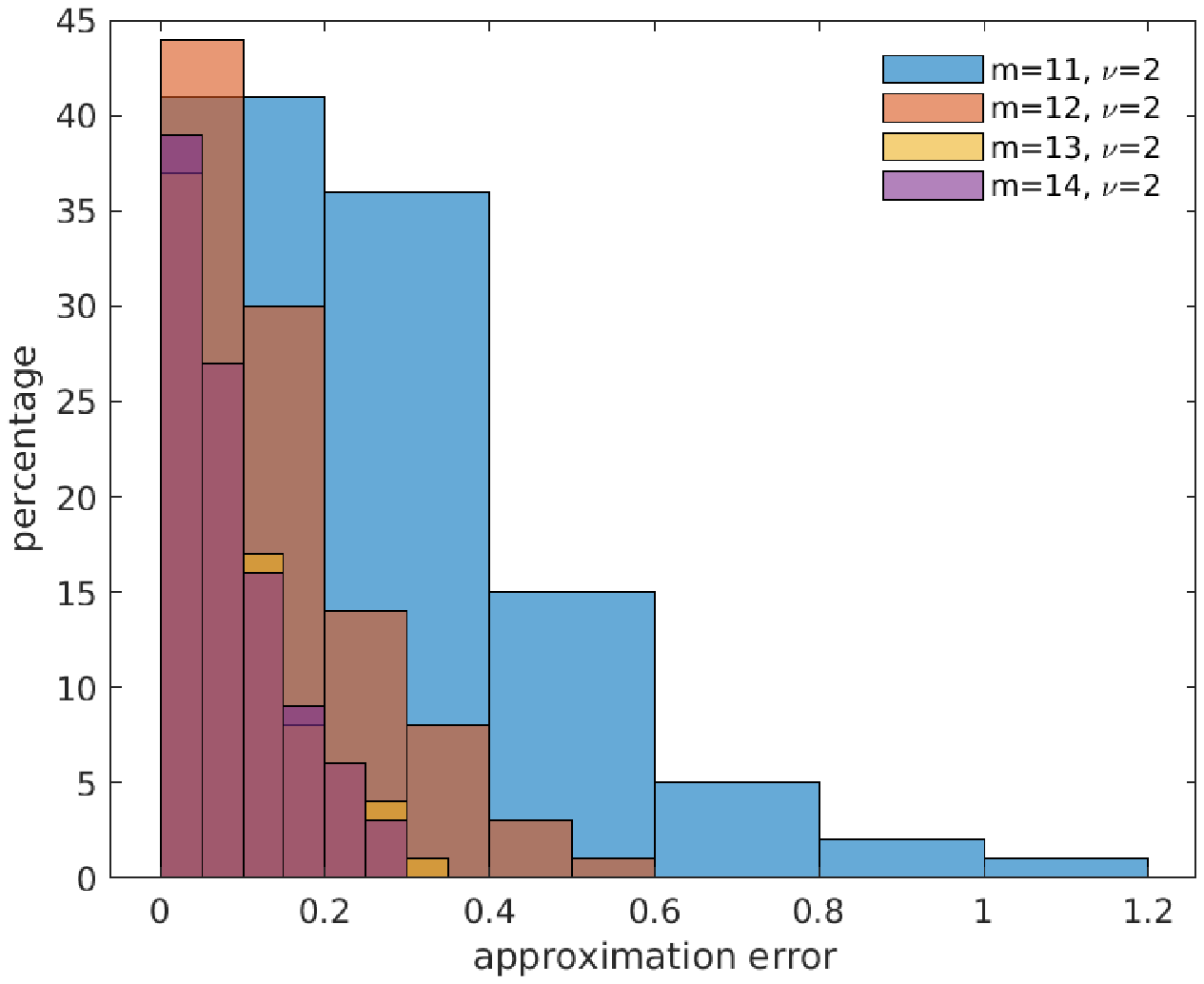}\\
\includegraphics[width=0.30\textwidth]{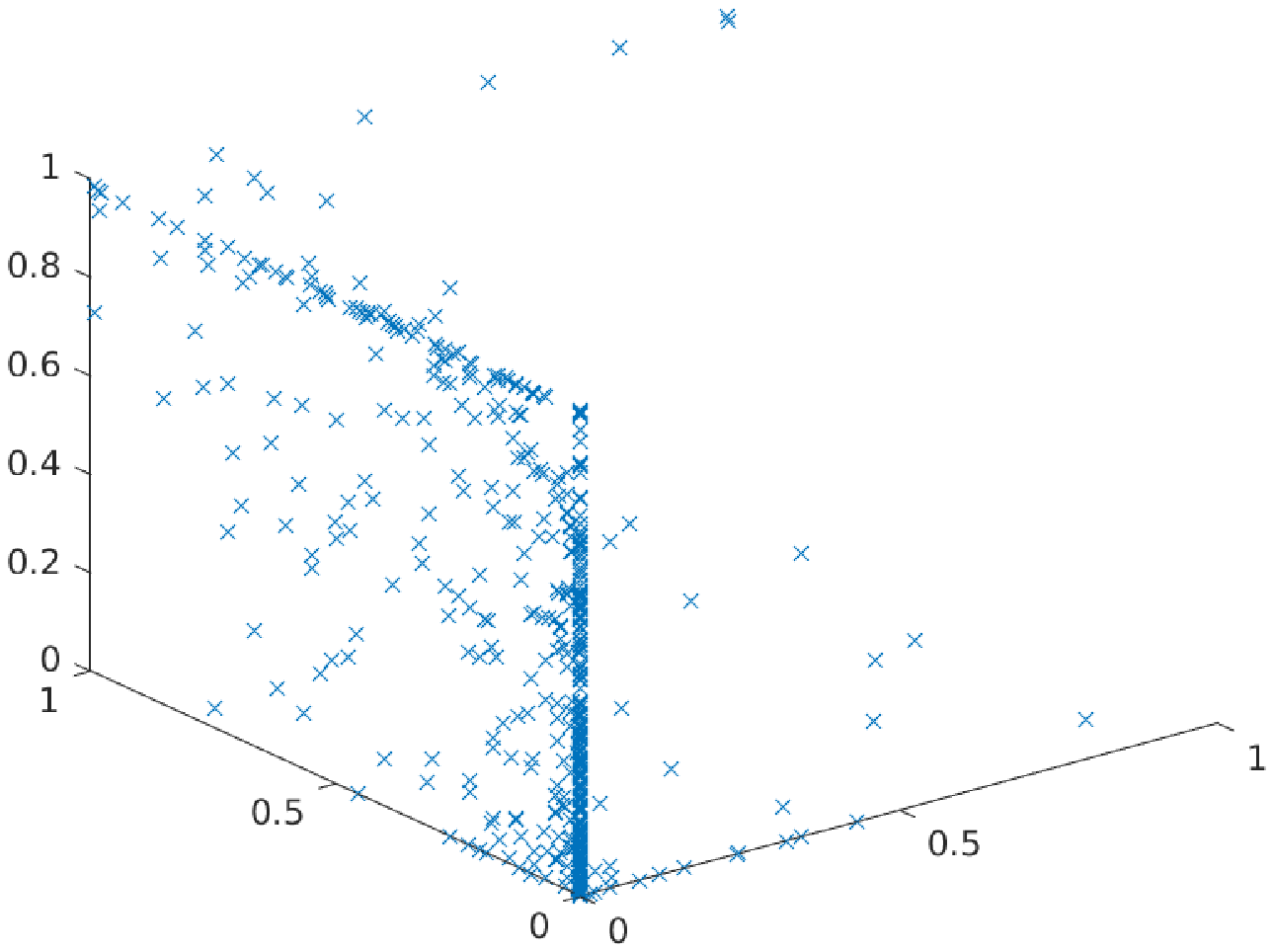}%
\includegraphics[width=0.30\textwidth]{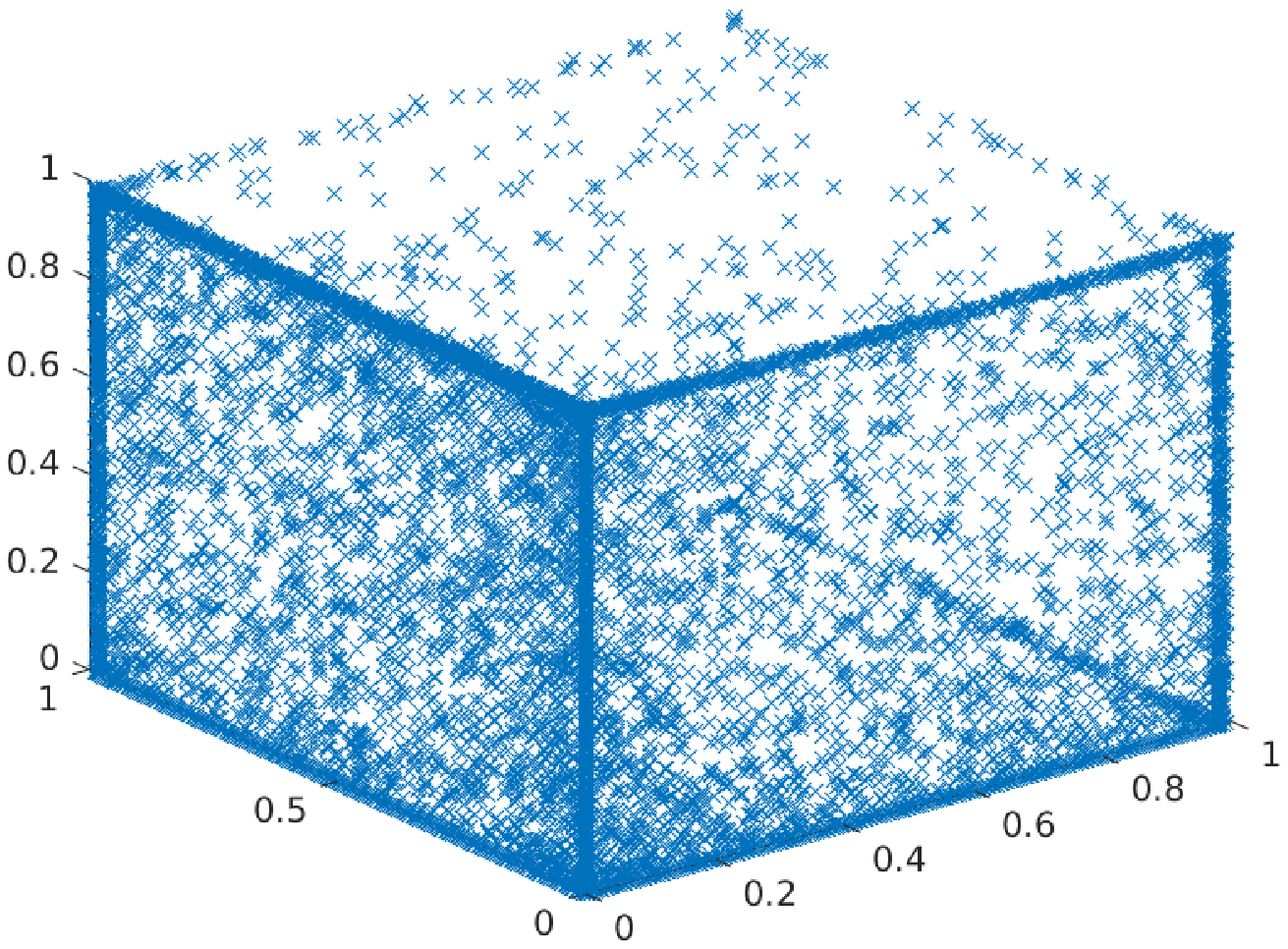}%
\includegraphics[width=0.30\textwidth]{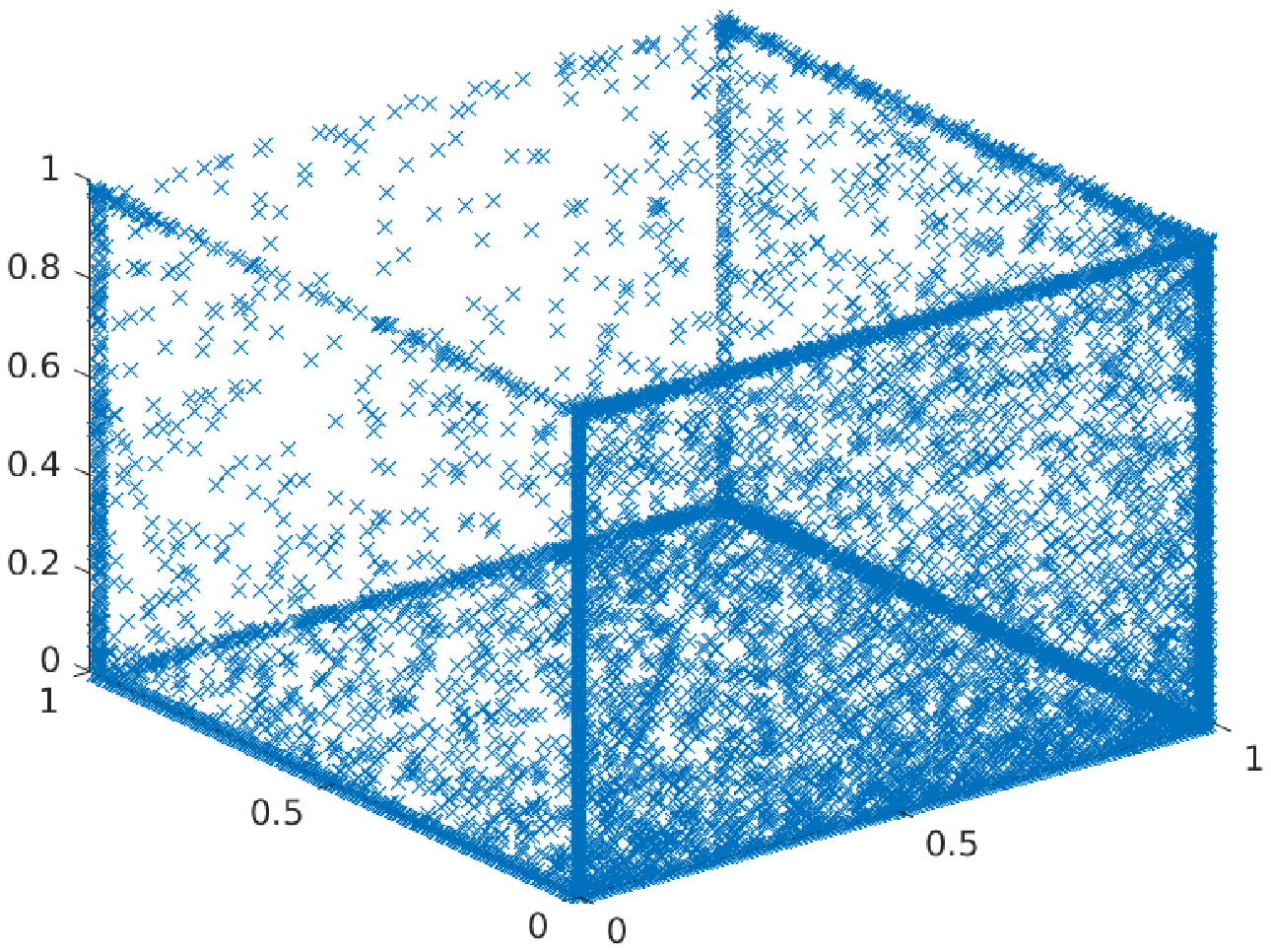}%
\caption{First row: Histrograms of the convergence of the approximation error over 100 randomly generated shallow nets (left) and deep nets (right). The approximation parameters are $N=60000$, $\nu$, and $m$ are given in the plot. Second row: Distribution of the MNIST data projected to three dimensional hyperfaces in dimensions $(1,2,3)$, $(50,80,150)$, and $(220,300,350)$.}
\label{fig:mnistapprox}
\end{figure}

\section*{Acknowledgements}

Josef Dick is partly supported by the Australian Research Council Discovery Project DP190101197. Michael Feischl is supported by the Deutsche Forschungsgemeinschaft (DFG, German Research Foundation) - Project-ID 258734477 - SFB 1173. The authors would like to thank Guoyin Li for pointing out the reference on perturbation analysis of optimization problems.

\bibliographystyle{plain}
\bibliography{literature}

%\bibitem{BD09} Baldeaux, Jan(5-NSW-SMS); Dick, Josef(5-NSW-SMS) QMC rules of arbitrary high order: reproducing kernel Hilbert space approach. (English summary) Constr. Approx. 30 (2009), no. 3, 495–527. 

\subsection*{Author's addresses}

\begin{itemize}
\item[] Josef Dick, School of Mathematics and Statistics, The University of New South Wales Sydney, Sydney NSW 2052, Australia; Email \email{josef.dick@unsw.edu.au}
\item[]
\item[] Michael Feischl, Institute for Analysis and Scientific Computing, TU Wien, Wiedner Hauptstra\ss e 8--10, 1040 Wien, Austria; Email: \email{michael.feischl@tuwien.ac.at}
\end{itemize}

\end{document}